\documentclass[10pt]{amsart}
\usepackage{epsf,latexsym,amsmath,amssymb,amscd,datetime}
\usepackage{amsmath,amsthm,amssymb,enumerate,eucal,url,calligra,mathrsfs,bbm}
\usepackage{rotating} 

\usepackage{blkarray} 

\usepackage{tikz,scalerel}
\usetikzlibrary{arrows,automata,calc,math}
\usepackage{upgreek}

\usepackage{imakeidx}     

\usepackage[outdir=./]{epstopdf}

\usepackage{graphicx}
\newenvironment{jfnote}{ \bgroup \color{blue} }{\egroup}

\usepackage[T1]{fontenc}

\usepackage{relsize}
\usepackage{tikz}
\usetikzlibrary{matrix,arrows,decorations.pathmorphing}
\usepackage{tikz-cd}
\usetikzlibrary{cd}

\usepackage[pdftex,colorlinks,linkcolor=blue,citecolor=brown]{hyperref}

\newcommand{\ignore}[1]{}

\DeclareMathOperator{\WalksMult}{\WalksMult}
\DeclareSymbolFont{Greekletters}{OML}{iwona}{m}{it}
\DeclareSymbolFont{greekletters}{OML}{iwona}{m}{n} 
\DeclareMathSymbol{\sfigamma}{\mathord}{Greekletters}{"0D}
\DeclareMathSymbol{\sfngamma}{\mathord}{greekletters}{"0D}
\DeclareMathSymbol{\sfgamma}{\mathord}{Greekletters}{"0D}
\newcommand{\naturals}{{\mathbb N}}

\newcommand{\mec}[1]{{\bf {#1}}}	

\makeatletter
 \newcommand{\linkdest}[1]{\Hy@raisedlink{\hypertarget{#1}{}}}
\makeatother
\def\centerarc[#1](#2)(#3:#4:#5)
    { \draw[#1] ($(#2)+({#5*cos(#3)},{#5*sin(#3)})$) arc (#3:#4:#5); }

\newcommand{\oldStuff}[1]{}

\DeclareMathOperator{\SHom}{\mathscr{H}\text{\kern -3pt {\calligra\large om}}\,}

\usepackage{mathrsfs}
\usepackage{amssymb}
\usepackage{dsfont}
\usepackage{verbatim}
\usepackage{url}

\theoremstyle{plain}
\ifundef\chapter
  {\newtheorem{theorem}{Theorem}[section]}
  {\newtheorem{theorem}{Theorem}[section]}
\ifundef\chapter
  { }
  {}

\ifundef\chapter
  { }
  {}

\ifundef\chapter
  { }
  {\numberwithin{equation}{section}
   
  }

\ifundef\chapter
  { }
  {\numberwithin{figure}{chapter}
}

\newtheorem{lemma}[theorem]{Lemma}
\newtheorem{proposition}[theorem]{Proposition}

\newtheorem*{theorem*}{Theorem}  
\newtheorem*{lemma*}{Lemma}
\newtheorem*{corollary*}{Corollary}
\newtheorem*{proposition*}{Proposition}

\theoremstyle{definition}
\newtheorem{definition}[theorem]{Definition}

\newtheorem{xca}{Exercise}[section]

\newtheorem{example}[theorem]{Example}

\newtheorem{remark}[theorem]{Remark}

\newcommand{\reals}{{\mathbb R}}

\newcommand{\integers}{{\mathbb Z}}

\DeclareMathAlphabet{\mathcal}{OMS}{cmsy}{m}{n}

\newcommand\cD{\mathcal{D}}

\newcommand\cG{\mathcal{G}}

\newcommand\cN{\mathcal{N}}

\newcommand\cR{\mathcal{R}}

\newcommand\cU{\mathcal{U}}

\def\from{\colon}

\def\eqdef{\overset{\text{def}}{=}}

\def\implies{\Rightarrow}

\DeclareRobustCommand
  \rddots{\mathinner{\mkern1mu\raise\p@
    \vbox{\kern7\p@\hbox{.}}\mkern2mu
    \raise4\p@\hbox{.}\mkern2mu\raise7\p@\hbox{.}\mkern1mu}}

\newcommand\xhookrightarrow[2][]{\ext@arrow 0062{\hookrightarrowfill@}{#1}{#2}}
\def\hookrightarrowfill@{\arrowfill@\lhook\relbar\rightarrow}

\begin{document}

\title[Example] 
{Urschel Nodal Domains via Perturbation Theory}


\author{Joel Friedman}
\address{Department of Computer Science, 
        University of British Columbia, Vancouver, BC\ \ V6T 1Z4, CANADA.}
\curraddr{}
\email{{\tt jf@cs.ubc.ca}}
\thanks{Research of the fist two authors
supported in part by an NSERC grant.}

\author{Tong Ling}
\address{Department of Computer Science,
        University of British Columbia, Vancouver, BC\ \ V6T 1Z4, CANADA.}
\email{{\tt toling@student.ubc.ca}}

\author{Soumyajit Saha}
\address{Institut de Recherche Math\'{e}matique Avanc\'{e}e (IRMA), 7 rue Ren\'{e} Descartes, 67084 Strasbourg, France}
\email{{\tt soumyajit.saha@unistra.fr}}
\thanks{The work is partially funded by the European Union. The third named author would like to thank Institut de Recherche Math\'{e}matique Avanc\'{e}e (IRMA), France, and the ERC (grant Acronym = InSpeGMos, Number = 101096550) for supporting the work.}



\keywords{}

\begin{abstract}
We prove several types of Courant nodal domain theorems for generalized
Laplacians on graphs, 
based on an invariant introduced by Urschel, which we call the
{\em Urschel number}, denoted ${\rm UN}(\mec f)$, of an eigenvector
$\mec f$.
We refine Urschel's invariant, and use perturbation techniques to
obtain some new results.
First, 
we show the existence of
mutually orthogonal eigenvectors, such that if the $k$-th eigenvalue
has multiplicity $m$, then for $0\le j\le m-1$,
${\rm UN}(\mec f_{k+j})\le k+\min(j,(m-1)-j)$.
Second, for a simple $k$-th eigenvalue, we classify the
zeroes of $\mec f_k$ as either {\em shallow} or {\em deep}; we obtain
a number of results that say, roughly speaking, the more
shallow vertices $\mec f_k$ has, the more control we have over our
new invariants based on Urschel's.

Our new invariants of an eigenvector, $\mec f_k$,
are a sequence of integers whose minimum
value is ${\rm UN}(\mec f_k)$ and whose maximum, denoted
${\rm UN}_{\max{}}(\mec f_k)$, is the maximum number of nodal domains of
any possible positive/negative signing or ``charge'' of
the zeroes of $\mec f_k$.
An example of our second type of result is that if
$\mec f_k$ has no deep vertices, then ${\rm UN}_{\max{}}(\mec f_k)\le k$.

We provide a number of examples to illustrate our main results, and
how they differ from the situation in analysis.  We also describe
a minor improvement of the Gladwell-Zhu theorem for an orthonormal
eigenbasis in the presence of eigenvalues of sufficient multiplicity.
\end{abstract}

\maketitle
\setcounter{tocdepth}{3}
\tableofcontents


\newcommand{\myred}{}

\section{Introduction}
\label{se_intro}

\newcommand{\MUN}{{\rm UN}_{\rm max}}

The main point of this article is to use perturbation theory
as a way of obtaining results of the type that Urschel
\cite{urschel} obtained; these concern nodal decompositions of
an eigenvector of a {\em generalized Laplacian} 
(also known as {\em Schr\"odinger operators})
on a graph.
Our results show that one can derive part (but not all) of Urschel's results
using perturbation theory;
we give an improvement of Urschel's results 
in the presence of multiple eigenvalues.
We also introduce a family of invariants based on Urschel's invariant,
and we obtain some new results on these invariants.
The proofs of our main theorems use only
very general facts about perturbation theory
and some simple
graph theory;
however, our proofs use algorithms
that involve matrix computations including pseudoinverses,
and hence may not be as simple as Urschel's
linear algebraic algorithm.

In any case, our perturbation theory sheds new light on
Urschel's work \cite{urschel}.

Let us recall the rough context.

If $f\from X\to\reals$ is a continuous map from a topological space
to the reals, then $f^{-1}(0)\subset X$ is called the
{\em nodal set of $f$}, and $X\setminus f^{-1}(0)$ --- or equivalently
$f^{-1}(\reals\setminus\{0\})$ --- decomposes into its connected components,
each of which is called a {\em nodal domain of $f$}.

Courant's nodal region theorem
says that if
$\cD\subset \reals^n$ is a domain with sufficiently smooth boundary,
and the eigenvalues of its Laplacian subject to Dirichlet conditions
are ordered 
$\lambda_1\le\lambda_2\le\cdots$,
then for any $k$, an eigenfunction, $f_k$, of eigenvalue $\lambda_k$,
satisfies
$$
{\rm ND}(f_k)\le k,
$$
where ${\rm ND}(f)$ denotes the number of nodal regions of $f$.

There is a lot of 
work in graph theoretic analogs of Courant's theorem
\cite{powers1988,
colinDeVerdiere1993,
friedman_geometric_aspects,
biyikoglu_et_al_textbook,
duval_reiner,
gladwell_zhu,
davies_et_al2001,
mckenzie_urschel,
urschel}.
The interest is in part due to
the useful heuristic to use $f_2$ (also called the Fiedler vector)
to produce a good graph separator;
see also references in the textbook
\cite{biyikoglu_et_al_textbook}
and in \cite{powers1988,biyikoglu_et_al_laa_2004},
which go back to Fiedler's work in the 1970's.

If $G=(V_G,E_G)$ is a graph, and
$\mec f\from V_G\to\reals$,\footnote{We use
  bold face for functions $\mec f\from V_G\to\reals$ and often refer to
  them as {\em vectors}.
  }
one can analogously
define the nodal sets, nodal domains,
and ${\rm ND}(\mec f)$; see Section~\ref{se_main}.
For reasons to become clear, we will use the terminology
{\em strong nodal domains} and use ${\rm SND}(\mec f)$ 
instead of ${\rm ND}(\mec f)$.
This notion was done ``geometrically'' in 
\cite{friedman_geometric_aspects}\footnote{
  The definitions in \cite{friedman_geometric_aspects} replace each
  edge with a unit length interval, and there nodal sets make sense
  and nodal regions are defined topologically.
  },
however most authors prefer to work discretely, i.e., where nodal
domains can be defined purely with graph theory.\footnote{
  However, when working purely with graph theory, then nodal sets don't
  make much sense, and one often has to appear to more involved
  machinery, including interlacing theorems and the Lemma~5 of
  Duval-Reiner \cite{duval_reiner}.
  }
A graph $G$ has a positive definite {\em graph Laplacian}, $\Delta_G$,
but one can define
a much broader class of {\em generalized Laplacians}.

The analog of Courant's theorem does not hold for graphs: indeed,
if $G$ is a star graph, then a ``typical'' eigenvector $\mec f_2$ of the
$\lambda_2$ graph Laplacian eigenvalue satisfies
$$
{\rm SND}(\mec f_2)=n-1.
$$
However, Friedman \cite{friedman_geometric_aspects} proved that
\begin{equation}\label{eq_friedmans_nodal_domain_theorem}
{\rm SND}(\mec f_k)\le k
\quad
\mbox{if $\mec f_k$ is nowhere zero,}
\end{equation} 
i.e., $f(v)\ne 0$ for all $v\in V_G$.
This is therefore a version of Courant's theorem for graphs.

To improve upon \eqref{eq_friedmans_nodal_domain_theorem} and get a
result that remains close to Courant's theorem,
one can define the notion of a {\em weak nodal domain}
of a function $\mec f\from V_G\to\reals$, and the number of weak nodal domains
of $\mec f$ which we denote ${\rm WND}(\mec f)$; see Section~\ref{se_main}.
This notion has the property that
$$
{\rm WND}(\mec f)\le {\rm SND}(\mec f),
$$
with equality whenever $\mec f$ doesn't vanish on any vertex.

Davies et al.\ \cite{davies_et_al2001} proved that
\begin{equation}\label{eq_weak_nodal_domain_courant_theorem}
{\rm WND}(\mec f_k)\le k
\end{equation} 
for any $k$-th eigenfunction, $\mec f_k$.
Moreover, this also holds for eigenvectors and eigenvalues of any
{\em generalized Laplacian} (see 
Definition~\ref{de_generalized_laplacian}).\footnote{
  Friedman worked only with Laplacians
  \cite{friedman_geometric_aspects}, but these results and
  proofs 
  hold for generalized Laplacians as well.
  }
This improves Friedman's result \eqref{eq_friedmans_nodal_domain_theorem},
since if $\mec f_k$ is nowhere zero, then
${\rm WND}(\mec f_k)={\rm SND}(\mec f_k)$.

Hence \eqref{eq_weak_nodal_domain_courant_theorem} is a close analog of
Courant's theorem.
Davies et al.\ \cite{davies_et_al2001} also proved that if $\lambda_k$
is an eigenvalue of multiplicity $m$, so that
\begin{equation}\label{eq_kth_eigenvalue_has_multiplicity_m_intro}
\lambda_{k-1}<\lambda_k=\lambda_{k+1}=\ldots=\lambda_{k+m-1}
<\lambda_{k+m},
\end{equation} 
then
\begin{equation}\label{eq_strong_nodal_domain_courant_theorem}
{\rm SND}(\mec f_k)\le k+m-1 ,
\end{equation} 
again for generalized Laplacians.

Hence Courant's theorem, ${\rm SND}(\mec f_k)\le k$,
holds if $\mec f_k$ is nowhere zero or if $\lambda_k$ is of multiplicity one
(but not on a star graph with $n\ge 4$ vertices).

We also remark that ${\rm WND}(\mec f)$ and ${\rm SND}(\mec f)$ depend only on
the subset of vertices where $\mec f$'s values are positive, negative, and
zero.

In \cite{urschel}, Urschel improves on
\eqref{eq_weak_nodal_domain_courant_theorem} as follows.
For a function, $\mec f\from V_G\to\reals$, a {\em signing of $\mec f$}
is a function $\tilde{\mec f}$ that (1) is nowhere zero, and
(2) agrees with the sign of $\mec f$ --- positive or negative 
--- wherever $\mec f$ is nonzero.
Urschel defined an invariant,
which we call the
{\em Urschel number} of a function $\mec f\from V_G\to\reals$, which is
\begin{equation}\label{eq_define_Urschel_number}
{\rm UN}(\mec f) = \min\bigl\{ {\rm SND}(\tilde{\mec f})\ \bigm| 
\ \mbox{$\tilde{\mec f}$ is a signing of $\mec f$} \bigr\}.
\end{equation} 
We easily see that
\begin{equation}\label{eq_Urschel_number_between_nodal_counts}
{\rm WND}(\mec f) \le {\rm UN}(\mec f) \le {\rm SND}(\mec f)
\end{equation} 
provided that $\mec f$ does not vanish on any connected component of $G$.
Urschel \cite{urschel}
proved that for the $k$-th eigenvalue of a generalized Laplacian,
there exist $k$-th eigenfunctions, $\mec f_k$, such that
\begin{equation}\label{eq_Urschels_theorem}
{\rm UN}(\mec f_k)\le k;
\end{equation} 
moreover, if the corresponding eigenvalue, $\lambda_k$, is of
multiplicity $m$, i.e., 
\eqref{eq_kth_eigenvalue_has_multiplicity_m_intro} holds,
then Urschel proved that 
the set of such $\mec f_k$ with Urschel number at most $k$ is of positive
measure in the $m$-dimensional space of eigenvectors with eigenvalue 
$\lambda_k$.
In view of \eqref{eq_Urschel_number_between_nodal_counts},
\eqref{eq_Urschels_theorem} improves 
\eqref{eq_weak_nodal_domain_courant_theorem}.

We remark that if we define the {\em maximum Urschel number}
of a function as
\begin{equation}\label{eq_define_maximum_Urschel_number}
\MUN(\mec f) \eqdef  \max\{ {\rm SND}(\tilde{\mec f}) \ | \mbox{$\tilde{\mec f}$ is a 
signing of $\mec f$} \},
\end{equation} 
then
for the star on $n\ge 3$ vertex,
we have that for {\em any} eigenvector, $\mec f_2$,
of the (classical) Laplacian,
$$
\MUN(\mec f)
\ge (n+1)/2 .
$$
Hence Urschel's result does not hold if we replace ${\rm UN}$ with
$\MUN{}$ in \eqref{eq_Urschels_theorem}.

We also remark that if $G$ is a star on $n\ge 5$ vertices, and
$\Delta_G$ is its Laplacian, then when an eigenfunction, $\mec f$, 
of $\lambda_2$
has $\lfloor (n-1)/2\rfloor$ 
positive values at its $n-1$ leaves and the rest negative,
$$
{\rm UN}(\mec f) \ge n/2.
$$
Hence Urschel's result does not hold for $\mec f$ in some codimension $0$
subset of $\Delta_G$ eigenfunctions of eigenvalue $\lambda_2$.

The above two paragraphs give senses in which one cannot improve
on Urschel's result.
One way to improve Urschel's result is to notice that for a graph, $G$,
and a function
$\mec f\from V_G\to\reals$, if $\mec f$ vanishes on 
$\ell$ vertices, then there are
essentially $2^\ell$ possible signings of $\mec f$,
$$
\tilde{\mec f}_1,\ldots,\tilde{\mec f}_{2^\ell}
$$
(where we identify two signings with the same positive and negative
value).
Arranging the $\tilde{\mec f}_i$ in increasing number of strong nodal domains,
$$
{\rm SND}(\tilde{\mec f}_1)\le {\rm SND}(\tilde{\mec f}_2)\le \cdots \le
{\rm SND}(\tilde{\mec f}_{2^\ell}),
$$
we can define
$$
{\rm UN}_i(\mec f) = {\rm SND}(\tilde{\mec f}_i);
$$
hence
$$
{\rm UN}(\mec f) = {\rm UN}_1(\mec f)
\quad
\MUN(\mec f) = {\rm UN}_{2^\ell}(\mec f).
$$

When $\lambda_k$ is a simple eigenvalue of a generalized Laplacian,
$M$, our results apply to ${\rm UN}_2(\mec f)$ and --- under a certain
condition --- to $\MUN(\mec f)$.
When $\lambda_k$ is a multiple eigenvalues of a generalized Laplacian,
$M$, we get improved Urschel number bounds on an orthonormal eigenbasis
of the eigenspace of $\lambda_k$ (this was left open by Urschel).

In this article we take a different approach to Urschel's theorem,
using perturbation theory; we have three results
concerning a generalized Laplacian, $M$, on a graph, $G$.
Namely, we consider $M(\epsilon)=M+\epsilon M_1$ for appropriately
chosen $M_1$ such that $M(\epsilon)$ is a generalized Laplacian for
$|\epsilon|$ sufficiently small.
Our perturbation theory does two things.

First, if $\lambda_k$ is a simple eigenvalue of $M$, then we can
always find an $M_1$ such that
perturbed eigenvector, $\mec f_k(\epsilon)$, of $M(\epsilon)$, will
be nonzero at some places where $\mec f_k$ is zero.
In fact, by applying this to $\epsilon>0$ and $\epsilon<0$,
$\mec f_k(\epsilon)$ yields two non-equivalent partial signings of $\mec f_k$
such that the Urschel set of $\mec f_k(\epsilon)$ is a proper subset of
that of $\mec f_k$.
Applying this repeatedly gives
two different signings of $\mec f_k$, say $\tilde{\mec f}_k,\hat{\mec f}_k$, 
such that
$$
{\rm SND}(\tilde{\mec f}_k)\le k
\quad\mbox{and}\quad
{\rm SND}(\hat{\mec f}_k)\le k.
$$
Hence
$$
{\rm UN}_2(\mec f_k)\le k.
$$
Hence perturbation theory ``chooses'' a signing of $\mec f_k$.
Also, perturbation theory applied with $\epsilon$ both positive and
negative gives us a ``double result,'' in this case producing
two signings with nodal domain number at most $k$.
[Urschel's method may also yield two such signings, but we are not
sure of this.]\footnote{
  At some point Urschel's algorithm gives all vertices in his
  $V_{\hat i_1}$ a positive charge (third paragraph, page~70 of
  \cite{urschel}); perhaps a negative charge would
  yield a different signing(?).
  }
  
Second, if $\lambda_k$ has multiplicity $m\ge 2$ in $M$, then we prove
that an appropriate $M_1$ will separate the $m$ eigenvalues 
($\lambda_k=\cdots=\lambda_{k+m-1}$) into
individual eigenvalues.  We use this to conclude that there
are orthogonal eigenvectors $\mec f_k,\ldots,\mec f_{k+m-1}$ of $M$
such that
\begin{equation}\label{eq_f_j_un_le_j_multiple_eigenvalue}
{\rm UN}(\mec f_j)\le j
\end{equation} 
for $k\le j\le m-1$.
However, because perturbation allows us to consider the eigenvalues/vectors
of $M(\epsilon)=M+\epsilon M_1$ for both positive and negative 
values of $\epsilon$, we will also conclude that
$$
{\rm UN}(\mec f_{k+m-1})\le k, 
\ \ {\rm UN}(\mec f_{k+m-2})\le k+1, 
\ \ldots 
$$
Hence
\begin{equation}\label{eq_new_result_for_multiple_eigenvalue}
\mbox{for $0\le i\le m-1$},\quad
{\rm UN}(\mec f_{k+i}) \le k + \min(i,m-1-i) .
\end{equation} 
Hence again we get a ``double result,'' which this time means that
we get an improved bound on the Urschel numbers in
\eqref{eq_new_result_for_multiple_eigenvalue} once $i\ge m/2$.
Hence in both perturbation theory results above,
taking 
$\epsilon>0$ and $\epsilon<0$ give a ``double result.''

We note that Gladwell and Zhu \cite{gladwell_zhu} proved that there exist 
orthonormal eigenfunctions
$\mec f_1,\mec f_2,\ldots$ of $M$ such that ${\rm SND}(\mec f_j)\le j$.
Hence \eqref{eq_new_result_for_multiple_eigenvalue} gives
a new result, albeit not a stronger result,
since we are using Urschel numbers rather than strong nodal domain
numbers.
However, our result \eqref{eq_new_result_for_multiple_eigenvalue} has
lead us to give a mild improvement of the Gladwell-Zhu result
in some special cases of eigenvalues of sufficient multiplicity.

Another perturbation 
result in the case of a simple eigenvalue $\lambda_k$ is that we 
can prove that
$$
{\rm UN}_{2^s}(\mec f_k)\le k
$$
provided that $\mec f_k$ has $s$ {\em shallow Urschel vertices}:
that is, we call a $v\in V_G$ an {\em Urschel vertex} of $\mec f_k$
if $\mec f_k(v)=0$, and we subdivide Urschel vertices into
{\em shallow Urschel vertices} and {\em deep Urschel vertices}.
In particular, if all Urschel vertices are shallow, then
$$
\MUN(\mec f_k)\le k .
$$
We shall define this notion formally in
Section~\ref{se_main} and give a number of examples
of graphs and eigenvectors 
[However, this does not hold for a star on $n\ge 4$ vertices, for then
$\lambda_2$ is a multiple eigenvalue.]

The rest of this paper is organized as follows.
In Section~\ref{se_main} we give formal definitions and state
our main theorems.
In Sections~\ref{se_urschel_examples},
\ref{se_ladder_again}, and~\ref{se_shallow_deep} 
we will give some examples of
graphs where Urschel's theorems and ours are tight.
These three sections are not
necessary to the proofs of our main theorems; however, they provide
good intuition regarding our main theorems.
In Section~\ref{se_perturbation} we will review all the
perturbation theory we need for the proofs of our main theorems.
In Sections~\ref{se_gen_Laplacian_perturb}
and~\ref{se_gen_Laplacian_perturb_multiple} we prove our main theorems.
In Section~\ref{se_concluding}
we give some concluding remarks, including some minor
improvements to the Gladwell-Zhu result above when an eigenvalue
has sufficient multiplicity (depending on its position).
Finally, in Appendix~\ref{ap_details_of_ladders_and_split_paths} we prove
that ``ladders'' and ``paths with a double left end''
have the eigenvalues we claim they have
in Section~\ref{se_urschel_examples}.

We wish to thank Nalini Anantharaman for conversions regarding
Urschel's work and our methods.

\section{Main Results}
\label{se_main}

In this section we formalize some notation and summarize our main results.  

\subsection{Basic Notation}

We use $\reals$ to denote the reals.
If $A$ is a set, then $|A|$ denotes the cardinality of $A$,
and $\reals^A$ denotes the $\reals$-vector space of maps $A\to\reals$.
For $a\in A$ we use $\mec e_a\in\reals^A$ to denote the standard basis
vector at $a$, i.e., $\mec e_a(a')$ is $1$ if $a'=a$, and $0$ if
$a'\ne a$.
For $n\in\naturals=\{1,2,\ldots\}$, $[n]$ denotes $\{1,2,\ldots,n\}$,
and $\reals^n$ denotes $\reals^{[n]}$.

\newcommand{\realBA}[2]{\reals^{#1\subset#2}}
\newcommand{\realBVG}[1]{\reals^{#1\subset V_G}}

If $A$ is a set, we generally use bold face for vectors in $\reals^A$,
e.g.,
$\mec u\in\reals^A$.

If $B\subset A$, we use the notation
\begin{equation}\label{eq_reals_B_subset_A}
\realBA{B}{A} = {\rm Span}\{\mec e_b \ | \ b\in B\} \subset \reals^A ,
\end{equation} 
i.e.,
$$
\realBA{B}{A} = \{ \mec u\in\reals^A \ | \ 
\mbox{$\mec u(b')=0$ for $b'\in A\setminus B$} \}.
$$

If $A,B$ are sets, we use $\reals^{A\times B}$ to denote the set
of $A\times B$ real matrices, which can be viewed as functions
$A\times B\to\reals$, and for $a\in A$, $b\in B$,
we use $M(a,b)$ to denote the $(a,b)$-entry of $M$.
$M\in\reals^{A\times A}$ is {\em symmetric} if for all $a,a'\in A$,
$M(a,a')=M(a',a)$, i.e., $M^{\rm T}=M$ where $M^{\rm T}$ denotes the
transpose of $M$.

\subsection{Conventions on Graphs}

In this paper, unless specified otherwise, all graphs will be
finite, simple graphs.
Hence a graph is a pair
$G=(V_G,E_G)$, where 
$V_G$ is a finite set,
and $E_G$ is a set of unordered pairs $\{u,v\}$ with $u,v\in V_G$.
Hence we have $|E_G|\le\binom{|V_G|}{2}$.

Let $G=(V_G,E_G)$ be a simple graph.  We write $v\sim v'$ and say
{\em $v$ and $v'$ are adjacent} to mean
that $v,v'\in V_G$ and $\{v,v'\}\in E_G$.
A {\em walk in $G$}
is a nonempty sequence of vertices $(v_1,v_2,\ldots,v_m)$ such that
$v_i\sim v_{i+1}$ for all $i\in[m-1]$;
we refer to $v_1,v_m$ as the {\em endpoints} of the walk.
A {\em path} is a walk with distinct vertices.
We will use other common graph theory terminology, including:
induced subgraph, connected components, and adjacency matrix, $A_G$, of $G$;
we refer the reader to textbooks such as 
\cite{cioaba_murty,biyikoglu_et_al_textbook}.

\subsection{Matricies Supported on a Graph, Generalized Laplacians, and
Eigenvalues}

\begin{definition}\label{de_generalized_laplacian}
Let $G=(V_G,E_G)$ be a graph.  A matrix $M\in\reals^{V_G\times V_G}$
is {\em supported on $G$} if $M(u,v)\ne 0$ implies that either
$u=v$ or $\{u,v\}\in E_G$.
A {\em generalized Laplacian on $G$} refers to any symmetric matrix 
$M\in\reals^{V_G\times V_G}$ 
that is supported on $G$ and such that
for all $\{u,v\}\in E_G$ we have $M(u,v)<0$.
If $M$ is a generalized Laplacian on $G$, 
then since $M$ is symmetric it has $n$ real eigenvalues, which
we will arrange 
in non-decreasing order
\begin{equation}
\label{eq_eigenvalues_increasing_order}
\lambda_1=\lambda_1(M)\le 
\lambda_2=\lambda_2(M) \le \cdots \le 
\lambda_n=\lambda_n(M) .
\end{equation}
\end{definition}
If $\lambda$ is an eigenvalue of $M$, we use
$E(M,\lambda)$ to denote $\ker(M-I\lambda)$, i.e.,
the space of eigenvectors of $M$ with eigenvalue
$\lambda$.

At times we use $E$ to denote an arbitrary subspace of $\reals^{V_G}$
or $\reals^n$;
we do this when our intended application is a subspace of the form 
$E(M,\lambda)$, but the discussion holds for a general subspace $E$.


\begin{example}
The classical Laplacian, $\Delta_G$, is an important example
of a generalized Laplacian, given by
$\Delta_G=D_G-A_G$, where $A_G$ is the adjacency matrix of $G$,
and $D_G$ is the diagonal degree counting matrix, i.e.,
whose $(v,v)$ entry is $\deg_G(v)$, i.e., the degree of $v$ in $G$, 
i.e., the number of vertices in $G$ adjacent to $v$.
\end{example}

\subsection{(Discrete) Nodal Domains}

Let $G=(V_G,E_G)$ be a simple graph.  If $\mec f\in \reals^{V_G}$, then
we use $\mec f^{-1}_{>0}$ to denote
$$
\mec f^{-1}_{>0} = \{ v\in V_G \ | \ \mec f(v)>0 \},
$$
and similarly with $>0$ replaced with $\ge 0$, $<0$, $\le 0$, $=0$, $\ne 0$.
A {\em strong nodal domain} is a connected component of the
the subgraph of $G$ induced on the vertex set $\mec f^{-1}_{>0}$ or on
$\mec f^{-1}_{<0}$; we use ${\rm SND}(\mec f)$ --- the {\em strong nodal
domain number of $\mec f$} --- to denote the number of such
connected components; this is therefore the number of connected components
of the subgraph of $G$ obtained by discarding
(1) all vertices $u$ where $\mec f(u)=0$, and
(2) all edges $\{u,v\}$
where $\mec f(u),\mec f(v)$ are not both positive or both negative.
See Figure~\ref{fi_nodal_domains_examples}.

\begin{figure}[ht]
$$
\begin{array}{cc}
\begin{tikzpicture}[scale=0.25]
\filldraw (0,4) circle (5pt);
\node at (0,5.5) {$0$};
\filldraw (-6,0) circle (5pt);
\draw(0,4) to (-6,0);
\node at (-6,-1.5) {$+$};
\filldraw (-2,0) circle (5pt);
\draw(0,4) to (-2,0);
\node at (-2,-1.5) {$+$};
\filldraw (2,0) circle (5pt);
\draw(0,4) to (2,0);
\node at (2,-1.5) {$-$};
\filldraw (6,0) circle (5pt);
\draw(0,4) to (6,0);
\draw(-6,0) to (2,0);
\node at (6,-1.5) {$-$};
\node at (0,-3) {\Small A graph with the sign};
\node at (0,-4.2) {\Small pattern of a function};
\end{tikzpicture}
\quad & \quad
\begin{tikzpicture}[scale=0.25]
\filldraw (-6,0) circle (5pt);
\node at (-6,-1.5) {$+$};
\filldraw (-2,0) circle (5pt);
\node at (-2,-1.5) {$+$};
\filldraw (2,0) circle (5pt);
\node at (2,-1.5) {$-$};
\filldraw (6,0) circle (5pt);
\node at (6,-1.5) {$-$};
\draw (-6,0) to (-2,0);
\node at (0,-3) {\Small Strong nodal domains};
\end{tikzpicture}
\\
\begin{tikzpicture}[scale=0.25]
\filldraw (0,4) circle (8pt);
\node at (0,5.5) {$0$};
\filldraw (-6,0) circle (8pt);
\draw(0,4)[very thick] to (-6,0);
\draw (-6,0)[very thick] to (-2,0);
\node at (-6,-1.5) {$+$};
\filldraw (-2,0) circle (8pt);
\draw(0,4)[very thick] to (-2,0);
\node at (-2,-1.5) {$+$};
\filldraw (2,0) circle (5pt);
\draw(0,4) to (2,0);
\node at (2,-1.5) {$-$};
\filldraw (6,0) circle (5pt);
\draw(0,4) to (6,0);
\node at (6,-1.5) {$-$};
\node at (0,-3) {\Small One weak nodal domain};
\end{tikzpicture}
\quad & \quad
\begin{tikzpicture}[scale=0.25]
\filldraw (0,4) circle (8pt);
\node at (0,5.5) {$0$};
\filldraw (-6,0) circle (5pt);
\draw(0,4) to (-6,0);
\node at (-6,-1.5) {$+$};
\filldraw (-2,0) circle (5pt);
\draw(0,4) to (-2,0);
\node at (-2,-1.5) {$+$};
\filldraw (2,0) circle (8pt);
\draw(0,4)[very thick] to (2,0);
\node at (2,-1.5) {$-$};
\filldraw (6,0) circle (8pt);
\draw(0,4)[very thick] to (6,0);
\node at (6,-1.5) {$-$};
\node at (0,-3) {\Small Another weak nodal domain};
\end{tikzpicture}
\end{array}
$$
\caption{Depicted is a function on a graph; nodal domains care only
about the sign of a function, so we indicate this with $+,-,0$.
A {\em strong nodal domain} is a connected component of the graph obtained
by delete all vertices of sign $0$, and
all edges except those where the endpoints have the same sign.
We therefore get $3$ strong nodal domains.
A {\em weak nodal domain} is a connected component of the induced subgraph
on the vertices with $+,0$ function values, or one with values $-,0$.
Hence a weak nodal domain can combine strong nodal domains through vertices
of value $0$.
We depict the two weak nodal domains using larger vertices and edges.
Any function, $\mec f$, with the indicated sign pattern therefore has
${\rm SND}(\mec f)=3$, and ${\rm WND}(\mec f)=2$.
This is the notion of nodal domains used by most authors; 
Friedman \cite{friedman_geometric_aspects} keeps all edges, extends functions
linearly along the edges, and uses the notion of nodal domains akin to
that in analysis.
Notice that if $\mec f$ is nowhere zero, then there is no difference between
a weak and strong nodal domain of $\mec f$, since it is only vertices where
$\mec f$ is zero that allows strong nodal domains to coalesce into weak
nodal domains.
}
\label{fi_nodal_domains_examples}
\end{figure}
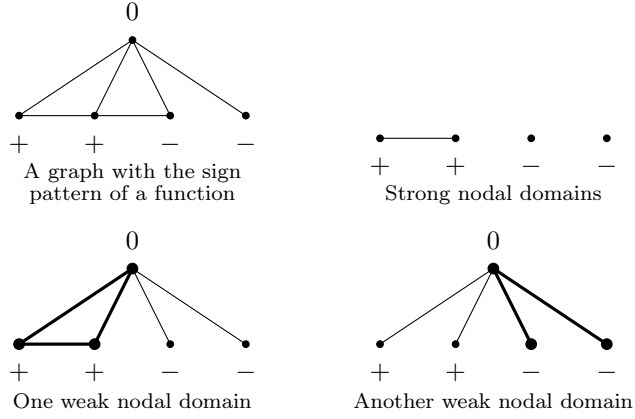

Notice that strong nodal domains care only about the sign, $+,-,0$
(i.e., positive, negative, and zero), of the function, not its particular
values.

Similarly, a
{\em weak nodal domain} is a connected component of the
the subgraph of $G$ induced on the vertex set $\mec f^{-1}_{\ge 0}$ or on
$\mec f^{-1}_{\le 0}$ 
that contains at least one vertex, $v$, at which
$\mec f(v)\ne 0$;
we use ${\rm WND}(\mec f)$ --- the {\em weak nodal
domain number of $\mec f$} --- to denote the number of such
connected components;
see Figure~\ref{fi_nodal_domains_examples} and see
Definitions~1 and~2 in \cite{davies_et_al2001}.\footnote{
  One has to be careful with this definition.
  }

Note that if a weak nodal domain contains one vertex of a strong
nodal domain, then it contains the entire strong nodal domain.
Hence each weak nodal domain is the union of one or more strong nodal domains,
plus any paths whose
interior vertices, $u$, satisfy
$\mec f(u)=0$;
therefore
$$
{\rm WND}(\mec f) \le {\rm SND}(\mec f) ,
$$
and equality always holds when $\mec f$ is everywhere nonzero
(i.e., $\mec f(v)\ne 0$ for all $v\in V_G$), since then 
$\mec f^{-1}_{>0}=\mec f^{-1}_{\ge 0}$, and 
$\mec f^{-1}_{<0}=\mec f^{-1}_{\le 0}$;
see Figure~\ref{fi_nodal_domains_examples}.

Notice that weak nodal domains --- like strong nodal domains ---
care only about the sign, $+,-,0$
(i.e., positive, negative, and zero), of the function at the vertex
set, not its particular
values.

If $\mec f_k$ is a $k$-th Laplacian eigenfunction of a graph, $G$,
then there are a number of
papers regarding upper bounds on ${\rm WND}(\mec f)$ and ${\rm SND}(\mec f)$;
in particular \cite{davies_et_al2001} results
\eqref{eq_weak_nodal_domain_courant_theorem}
and
\eqref{eq_strong_nodal_domain_courant_theorem}.

In this article we will study a third invariant, the {\em Urschel
number of $\mec f$} of a function $\mec f\from V_G\to\reals$.

\begin{definition}
Let $G$ be a simple graph, and $\mec f\from V_G\to \reals$ a function.
A {\em partial signing of $\mec f$} is any function 
$\tilde{\mec f}\from V_G\to\reals$
such that 
\begin{align*}
\forall v\in V_G, \quad 
& \mec f(v)>0 \ \implies\ \tilde{ \mec f}(v)>0 \\
\mbox{and}\quad & \mec f(v)<0 \ \implies\ \tilde{\mec f}(v)<0 ;
\end{align*}
we say that two partial signings are {\em equivalent}
if at each vertex their values are either both positive, both negative,
or both zero.
A {\em signing of $\mec f$} is a partial signing $\tilde{ \mec f}$ that is nonzero
at all vertices.
We use $\mbox{Signings}(\mec f)$ to denote the set of all signings of $\mec f$,
up to equivalence; hence, if $\mec f$ is zero at $\ell$ vertices, then
$\mbox{Signings}(\mec f)$ is a set of size $2^\ell$.
We define the {\em Urschel number of $\mec f$} to be
\begin{equation}\label{eq_Urschel_nodal_domain_number}
\mbox{UN}(\mec f) 
\eqdef \min \{ {\rm WND}(\tilde{\mec f}) | \tilde{\mec f}\in{\rm Signings}(\mec f) \}
= \min
\{ {\rm SND}(\tilde{\mec f}) | \tilde{\mec f}\in{\rm Signings}(\mec f) \} ,
\end{equation} 
which makes sense since WND and SND depend only the equivalence class
of a signing
(the equality in \eqref{eq_Urschel_nodal_domain_number} follows since 
${\rm WND}(\tilde{\mec f})={\rm SND}(\tilde{\mec f})$
for any nowhere zero function $\tilde{\mec f}\from V_G\to\reals_{\ne 0}$).
\end{definition}
See Figure~\ref{fi_Urschel_number} for an illustration.

\begin{figure}[ht]
$$
\begin{array}{cc}
\begin{tikzpicture}[scale=0.25]
\filldraw (0,4) circle (5pt);
\node at (0,5.5) {$0$};
\filldraw (-6,0) circle (5pt);
\draw(0,4) to (-6,0);
\node at (-6,-1.5) {$+$};
\filldraw (-2,0) circle (5pt);
\draw(0,4) to (-2,0);
\node at (-2,-1.5) {$+$};
\filldraw (2,0) circle (5pt);
\draw(0,4) to (2,0);
\node at (2,-1.5) {$-$};
\filldraw (6,0) circle (5pt);
\draw(0,4) to (6,0);
\node at (6,-1.5) {$-$};
\node at (0,-3) {\Small A star with the sign};
\node at (0,-4.2) {\Small pattern of a function};
\end{tikzpicture}
\quad & \quad
\begin{tikzpicture}[scale=0.25]
\filldraw (0,4) circle (5pt);
\node at (0,5.5) {$+$};
\filldraw (-6,0) circle (5pt);
\draw(0,4) to (-6,0);
\node at (-6,-1.5) {$+$};
\filldraw (-2,0) circle (5pt);
\draw(0,4) to (-2,0);
\node at (-2,-1.5) {$+$};
\filldraw (2,0) circle (5pt);
\draw(0,4) to (2,0);
\node at (2,-1.5) {$-$};
\filldraw (6,0) circle (5pt);
\draw(0,4) to (6,0);
\node at (6,-1.5) {$-$};
\node at (0,-4.2) {\Small One signing of the function};
\end{tikzpicture}
\\
\begin{tikzpicture}[scale=0.25]
\filldraw (0,4) circle (5pt);
\node at (0,5.5) {$+$};
\filldraw (-6,0) circle (5pt);
\draw(0,4) to (-6,0);
\node at (-6,-1.5) {$+$};
\filldraw (-2,0) circle (5pt);
\draw(0,4) to (-2,0);
\node at (-2,-1.5) {$+$};
\filldraw (2,0) circle (5pt);
\node at (2,-1.5) {$-$};
\filldraw (6,0) circle (5pt);
\node at (6,-1.5) {$-$};
\node at (0,-3) {\Small The strong (and weak) nodal};
\node at (0,-4.2) {\Small domains of the first signing};
\end{tikzpicture}
\quad & \quad
\begin{tikzpicture}[scale=0.25]
\filldraw (0,4) circle (5pt);
\node at (0,5.5) {$-$};
\filldraw (-6,0) circle (5pt);
\node at (-6,-1.5) {$+$};
\filldraw (-2,0) circle (5pt);
\node at (-2,-1.5) {$+$};
\filldraw (2,0) circle (5pt);
\draw(0,4) to (2,0);
\node at (2,-1.5) {$-$};
\filldraw (6,0) circle (5pt);
\draw(0,4) to (6,0);
\node at (6,-1.5) {$-$};
\node at (0,-3) {\Small A second signing and};
\node at (0,-4.2) {\Small its nodal domains};
\end{tikzpicture}
\end{array}
$$
\caption{Depicted is a function on a star graph.  There are two signings
of this function, each with 3 nodal domains, and hence the Urschel number
of this function is 3.
Notice that for any signing, or any nowhere 0 function, the number of
strong and weak nodal domains are the same.
}
\label{fi_Urschel_number}
\end{figure}
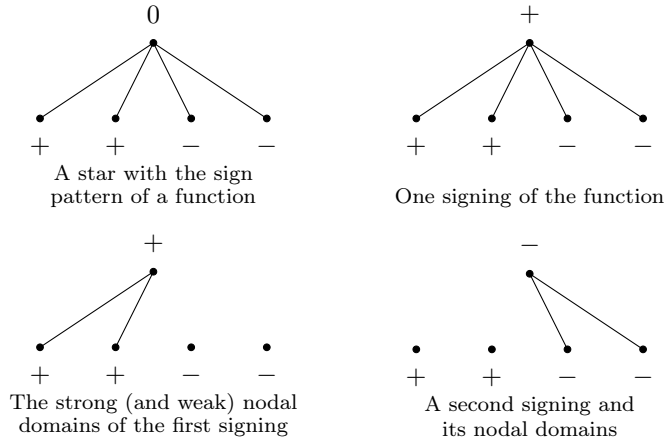

It is not hard to see\footnote{
  Urschel didn't formally prove this, so we do so here:
  the proof is an easy argument on induction on the number of vertices,
  $k$, at which $\mec f$ is zero.  In the base case $k=0$, clearly 
  \eqref{eq_weak_less_than_Urschel_less_than_strong} 
  holds with equality everywhere.
  For the inductive claim, we take any $u$ with $\mec f(u)=0$; by assumption
  there is a path from $u$ to a vertex $v$ with $\mec f(v)\ne 0$; hence there
  are adjacent $u',v'$ along this path where $\mec f(u')=0$ and $\mec f(v')\ne 0$;
  we let $\tilde{\mec f}$ be the partial signing that agrees everywhere with
  $\mec f$ except at $u'$ where we set $\tilde{\mec f}(u')=\mec f(v)$.  
  It is an easy exercise to check that ${\rm WND}(\mec f)\le{\rm WND}(\tilde{\mec f})$ 
  and
  ${\rm SND}(\tilde{\mec f})\le{\rm SND}(\mec f)$.
  Now we use induction. 
  }
\begin{equation}\label{eq_weak_less_than_Urschel_less_than_strong}
{\rm WND}(\mec f)\le {\rm UN}(\mec f) \le {\rm SND}(\mec f)
\end{equation} 
provided that $\mec f$ is nowhere zero on each 
connected component of $G$,\footnote{
  If $\mec f$ is the zero function, then ${\rm SND}(\mec f)=0$ and
  ${\rm UN}(\mec f)$ is the number of connected components of $G$.
  Hence it can happen that ${\rm UN}(\mec f)>{\rm SND}(\mec f)$ without assumptions
  on $\mec f$.
  }
and, moreover, equality holds
(in both inequalities in \eqref{eq_weak_less_than_Urschel_less_than_strong})
when $\mec f$ is nonzero at all vertices.\footnote{
  since then ${\rm WND}(\mec f)={\rm SND}(\mec f)$.
  }

\subsection{Refined Urschel Numbers}

Here we recall the definition of ${\rm UN}_i(\mec f)$ given in 
Section~\ref{se_intro}.
Say that 
$\mec f\from V_G\to\reals$ and that $\mec f$ vanishes on
$\ell$ vertices.  Then ${\rm Signings}(\mec f)$ consists of
$2^\ell$ equivalence classes of signings of $\mec f$.
Choose a representative of each:
$$
\tilde{\mec f}_1,\ldots,\tilde{\mec f}_{2^\ell}
$$
which we
arrange in increasing number of strong nodal domains,
i.e.,
$$
{\rm SND}(\tilde{\mec f}_1)\le {\rm SND}(\tilde{\mec f}_2)\le \cdots \le
{\rm SND}(\tilde{\mec f}_{2^\ell}).
$$
We define the {\em $i$-th Urschel number of $\mec f$} to be
$$
{\rm UN}_i(\mec f) = {\rm SND}(\tilde{\mec f}_i);
$$
hence
$$
{\rm UN}(\mec f) = {\rm UN}_1(\mec f).
$$
We also define the {\em maximum Urschel number of $\mec f$} to be
$$
\MUN(\mec f) = {\rm UN}_{2^\ell}(\mec f),
$$
or, in other words,
$$
\MUN(\mec f)
= 
\max \{ {\rm WND}(\tilde{\mec f}) | \tilde{\mec f}\in{\rm Signings}(\mec f) \}
= \max
\{ {\rm SND}(\tilde{\mec f}) | \tilde{\mec f}\in{\rm Signings}(\mec f) \} .
$$
In Figure~\ref{fi_i_th_Urschel_number_example},
we illustrate the $i$-th Urschel numbers of a function
$\mec f$ with two zeroes.

\begin{figure}[ht]
$$
\begin{array}{cc}
\begin{tikzpicture}[scale=0.25]
\filldraw (-2,0) circle (5pt);
\node at (-2,-1.5) {$+$};
\filldraw (2,0) circle (5pt);
\node at (2,-1.5) {$0$};
\filldraw (6,0) circle (5pt);
\node at (6,-1.5) {$0$};
\draw(-2,0) to (6,0);
\node at (2,-3.2) {\Small A function $\mec f$ on a path graph.};
\node at (2,-4.4) {\Small We indicate the sign pattern of $\mec f$.};
\end{tikzpicture}
& 
\begin{tikzpicture}[scale=0.25]
\filldraw (-2,0) circle (5pt);
\node at (-2,-1.5) {$+$};
\filldraw (2,0) circle (5pt);
\node at (2,-1.5) {$+$};
\filldraw (6,0) circle (5pt);
\node at (6,-1.5) {$+$};
\draw(-2,0) to (6,0);
\node at (2,-3.2) {\Small ${\rm UN}(\mec f)={\rm UN}_1(\mec f)=1$.};
\end{tikzpicture}
\end{array}
$$
$$
\begin{array}{ccc}
\begin{tikzpicture}[scale=0.25]
\filldraw (-2,0) circle (5pt);
\node at (-2,-1.5) {$+$};
\filldraw (2,0) circle (5pt);
\node at (2,-1.5) {$+$};
\filldraw (6,0) circle (5pt);
\node at (6,-1.5) {$-$};
\draw(-2,0) to (6,0);
\node at (2,-3.2) {\Small ${\rm UN}_2(\mec f)=2$.};
\end{tikzpicture}
&
\begin{tikzpicture}[scale=0.25]
\filldraw (-2,0) circle (5pt);
\node at (-2,-1.5) {$+$};
\filldraw (2,0) circle (5pt);
\node at (2,-1.5) {$-$};
\filldraw (6,0) circle (5pt);
\node at (6,-1.5) {$-$};
\draw(-2,0) to (6,0);
\node at (2,-3.2) {\Small ${\rm UN}_3(\mec f)=2$.};
\end{tikzpicture}
&
\begin{tikzpicture}[scale=0.25]
\filldraw (-2,0) circle (5pt);
\node at (-2,-1.5) {$+$};
\filldraw (2,0) circle (5pt);
\node at (2,-1.5) {$-$};
\filldraw (6,0) circle (5pt);
\node at (6,-1.5) {$+$};
\draw(-2,0) to (6,0);
\node at (2,-3.2) {\Small $\MUN(\mec f)={\rm UN}_4(\mec f)=3$.};
\end{tikzpicture}
\end{array}
$$
\caption{Depicted is a function, $\mec f$ on a ``path graph'' with two
zeros.  $\mec f$ therefore has four signings, and we have
${\rm UN}_1(\mec f)=1$
${\rm UN}_2(\mec f)={\rm UN}_3(\mec f)=2$, and 
${\rm UN}_4(\mec f)=3$.
}
\label{fi_i_th_Urschel_number_example}
\end{figure}
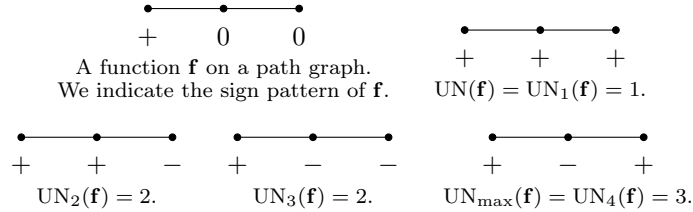

\subsection{Nodal Domain Numbers of Eigenvectors of Laplacians and
Generalized Laplacians}

There is a large literature regarding the number of nodal domains
of Laplacian and generalized Laplacian eigenvectors, and in
particular analog's of Courant's theorem, that states that if
$D\subset\reals^n$ is a domain, with Dirichlet Laplacian eigenvalues
$\lambda_1\le\lambda_2\le \cdots$ and corresponding
mutually orthogonal eigenfunctions $f_1,f_2,\ldots$,
then $f_k$ has at most $k$ nodal domains;
see 
\cite{friedman_geometric_aspects,
colinDeVerdiere1993,duval_reiner,davies_et_al2001,gladwell_zhu},
and many earlier references in
\cite{biyikoglu_et_al_laa_2004,biyikoglu_et_al_textbook},
beginning with the pioneering work of Fiedler in the early 1970's.

In this article, 
we will use the following version of Courant's
theorem.

\begin{theorem}\label{th_friedmans_nodal_region_theorem}
Let $M$ be a generalized Laplacian of a graph, $G$, and
let
$\mec f_1,\mec f_2,\ldots,\mec f_n$ be an orthonormal set of 
eigenvectors corresponding
to eigenvalues $\lambda_1\le \lambda_2\le \cdots\le \lambda_n$
of $M$.
Then if $\mec f_k$ is nowhere zero, then $\mec f_k$ has at most $k$ nodal
regions.
\end{theorem}
This was proven by Friedman \cite{friedman_geometric_aspects} for the
classical Laplacian $\Delta_G$; a minor error there is easy to fix;
it is easy to see that the methods and results of 
\cite{friedman_geometric_aspects} hold for generalized Laplacians
as well as the classical one.
As mentioned in the introduction, the above theorem is also
implied by \eqref{eq_weak_nodal_domain_courant_theorem}, proven in
\cite{davies_et_al2001}.

For bibliographical reference, we mention that there are two approaches
to the above theorem and to nodal domain theorems in general.
There is a ``geometric'' version
in \cite{friedman_geometric_aspects}, which is simple and closely
follows the proof in analysis.
However, most other authors 
\cite{
colinDeVerdiere1993,duval_reiner,davies_et_al2001,gladwell_zhu}
(see also those in \cite{biyikoglu_et_al_laa_2004,biyikoglu_et_al_textbook})
work ``discretely,'' using graph theory, and proofs of nodal domain
theorems use the powerful
Lemma~5
of Duval-Reiner \cite{duval_reiner}\footnote{This article also had a
  minor error in its main theorem.
  }, quoted as
Lemma~1 in \cite{davies_et_al2001},
Lemma~2 in \cite{gladwell_zhu}.

\subsection{Urschel's Theorem}

Urschel's theorem \cite{urschel} proves the following theorem.

\begin{theorem}
Let $M$ be a generalized Laplacian, whose eigenvalues are arranged as 
$\lambda_1\le \cdots\le \lambda_n$.
Say that $\lambda_{k-1}<\lambda_k=\cdots=\lambda_{k+m-1}<\lambda_{k+m}$,
and let $E(M,\lambda_k)$ be the subspace of vectors $\mec f$ satisfying
$M\mec f=\lambda_k \mec f$.  Then
\begin{enumerate}
\item 
there exists $\mec f_k \in E(M,\lambda_k)$ such that
${\rm UN}(\mec f_k)\le k$; and
\item 
the set of such $\mec f_k$ is a set of positive measure 
(i.e., of codimension $0$) in $E(M,\lambda_k)$.
\end{enumerate}
\end{theorem}

\subsection{The Urschel Nodal Domain Number of Eigenvalues of the Star}

The following example is important to keep in mind.

\begin{example}\label{ex_star_Urschel_example}
Let $G$ be the star on $n$ vertices, therefore having
a {\em centre vertex} of degree $n-1$, and $n-1$ leaves (each with an
edge to the centre).  It is easy to see that its Laplacian eigenvalues 
(i.e., of $M=\Delta_G$, the classical graph Laplacian) are
$$
0=\lambda_1,
\ 1 = \lambda_2=\cdots=\lambda_{n-1},
\ n = \lambda_n,
$$
and the eigenspace for $\lambda=1$ is the set of functions, $\mec f$,
such that $\mec f$ vanishes at the centre, and the sum of $\mec f$'s values at the
leaves is $0$.
It follows that if $\mec f$ is nowhere zero except at the centre, and is
hence positive at $m_+$ leaves and negative at $m_-$ leaves
(so $m_+ + m_- = n-1$), then we easily see that
$$
{\rm UN}(\mec f) = 1 + \min(m_+,m_-).
$$
Hence the $\mec f\in E(\Delta_G,1)$ that are nowhere zero at the leaves
and satisfy ${\rm UN}(\mec f)=2$ are those
with exactly one positive value at a leaf and all others negative,
or vice versa.
\end{example}

It is not hard to see\footnote{
  Indeed,
  the probability that a $\cN(0,1)$ random variable is larger than
  $m$ is bounded above and below by a constant times $e^{-m^2/2}/m$.
  If each of $x_1,\ldots,x_{n-1}$ is an independent Gaussian, then their
  mean is distributed like an $\cN(0,1/\sqrt{n-1})$, and is therefore
  at most $n^{-1/3}$ in absolute value with probability tending to $1$.
  But the probability that a given $x_i$ is greater than $n^{-1/3}$
  or less than $-n^{-1/3}$ tends to $1/2$.
  Hence the probability that exactly one of $n$ of the random variables,
  normalized by the mean, tends to $0$.
  }
that if $x_1,\ldots,x_{n-1}$ are Gaussian
random variables, then the probability that one of $x_i$ is
above the mean and $n-1$ are below the mean, or vice versa,
tends to $0$ as $n\to\infty$.
In this sense a ``typical'' eigenvector in
$E(\Delta_G,1)$, built by choosing $x_1,\ldots,x_{n-1}$ as Gaussian
and subtracting the mean from each $x_i$, has a small probability
of having its Urschel number equal to $2$ (as $n\to\infty$).

\subsection{Urschel Vertices}

Our perturbation theory around simple eigenvalues are tied to
what we call {\em Urschel vertices}, which we now define.
[We define them in the context of an arbitrary subspace,
$E\subset\reals^{V_G}$, although in this article we only care
about the case where $E$ is an eigenspace of an eigenvalue of a
generalized Laplacian.]

\begin{definition}
Let $G$ be a graph, and
$E\subset\reals^{V_G}$ be an arbitrary subspace of positive dimension.
We say that $v\in V_G$ is an
{\em Urschel vertex for $E$} if
for all $\mec f\in E$ we have $\mec f(v)=0$; otherwise we say that $v$ is a
{\em non-Urschel vertex}.  We use
$$
{\rm Urschel}(E), \ {\rm NonUrschel}(E)
$$
to denote the sets of Urschel and non-Urschel vertices of $E$.
We say that a vertex $v\in{\rm Urschel}(E)$ is a {\em shallow Urschel
vertex of $E$} if $v$ is adjacent to some non-Urschel vertex of $E$;
otherwise we say that $v$ is a {\em deep Urschel vertex of $E$};
we use
$$
{\rm Shallow}(E),\ {\rm Deep}(E)
$$
to denote the sets of shallow and deep Urschel vertices.
\end{definition}
Hence if $\dim(E)=1$, a vertex, $v$, is an Urschel vertex if $\mec f(v)=0$
for any nonzero $\mec f\in E$.
If $\dim(E)\ge 2$, then for all $v\in V_G$ there exists an $\mec f\in E$
such that $\mec f(v)=0$; hence $v$ is a non-Urschel vertex if the $\mec f\in E$
with $\mec f(v)=0$ are a subspace of dimension $\dim(E)-1$.

In Figure~\ref{fi_shallow_deep_urschel_illustration} we depict a 
two-dimensional space, $E$, and its non-Urschel and Urschel vertices,
both shallow and deep.
\begin{figure}[ht]
$$
\begin{tikzpicture}[scale=0.20]
\tikzmath{ 
  int \j ;
}
\foreach \i in {1,...,3} {
  \filldraw (0,6 - \i * 3) circle (5pt);
  \filldraw (10,6 - \i * 3) circle (5pt);
  \tikzmath{
    \j = \i + 3;
  }
}
\node[anchor=east]  at (-1,3) {\Small $\mec f(v_1)=a$};
\node[anchor=east] at (-1,0) {\Small $\mec f(v_2)=0$};
\node[anchor=east] at (-1,-3) {\Small $\mec f(v_3)=-a$};
\node[anchor=west] at (11,3) {\Small $\mec f(v_4)=b$};
\node[anchor=west] at (11,0) {\Small $\mec f(v_5)=0$};
\node[anchor=west] at (11,-3) {\Small $\mec f(v_6)=-b$};
\draw (0,3) to (0,-3);
\draw (10,3) to (10,-3);
\filldraw (5,5) circle (5pt);
\node at (5,6.5) {\Small $\mec f(v_7)=0$};
\filldraw (5,0) circle (5pt);
\node at (5,1.5) {\Small $\mec f(v_8)=0$};
\filldraw (5,-5) circle (5pt);
\node at (5,-6.5) {\Small $\mec f(v_9)=0$};
\draw (0,0) to (5,5);
\draw (10,0) to (5,5);
\draw (0,0) to (5,0);
\draw (10,0) to (5,0);
\draw (0,0) to (5,-5);
\draw (10,0) to (5,-5);
\end{tikzpicture}
\qquad
\begin{tikzpicture}[scale=0.20]
\tikzmath{ 
  int \j ;
}
\foreach \i in {1,...,3} {
  \filldraw (0,6 - \i * 3) circle (5pt);
  \filldraw (10,6 - \i * 3) circle (5pt);
  \tikzmath{
    \j = \i + 3;
  }
}
\node[anchor=east]  at (-1,3) {\Small non-Urschel};
\node[anchor=east] at (-1,0) {\Small shallow};    
\node[anchor=east] at (-1,-3) {\Small non-Urschel};
\node[anchor=west] at (11,3) {\Small non-Urschel};
\node[anchor=west] at (11,0) {\Small shallow};    
\node[anchor=west] at (11,-3) {\Small non-Urschel};
\draw (0,3) to (0,-3);
\draw (10,3) to (10,-3);
\filldraw (5,5) circle (5pt);
\node at (5,6.5) {\Small deep};      
\filldraw (5,0) circle (5pt);
\node at (5,1.5) {\Small deep};      
\filldraw (5,-5) circle (5pt);
\node at (5,-6.5) {\Small deep};      
\draw (0,0) to (5,5);
\draw (10,0) to (5,5);
\draw (0,0) to (5,0);
\draw (10,0) to (5,0);
\draw (0,0) to (5,-5);
\draw (10,0) to (5,-5);
\end{tikzpicture}
$$
\caption{
On the left we show a two dimensional subspace of functions, $E$, on the
vertices of a graph, where $a,b$ vary over $\reals$.
On the right, we show which vertices are non-Urschel, and of the
Urschel vertices we show which are shallow Urschel vertices, and which
are deep.
The shallow Urschel vertices are those Urschel vertices adjacent 
(i.e., distance 1) to at
least one
non-Urschel vertex, the deep Urschel vertices are only adjacent to
Urschel vertices (i.e., distance at least 2 to the set of Urschel vertices).
}
\label{fi_shallow_deep_urschel_illustration}
\end{figure}
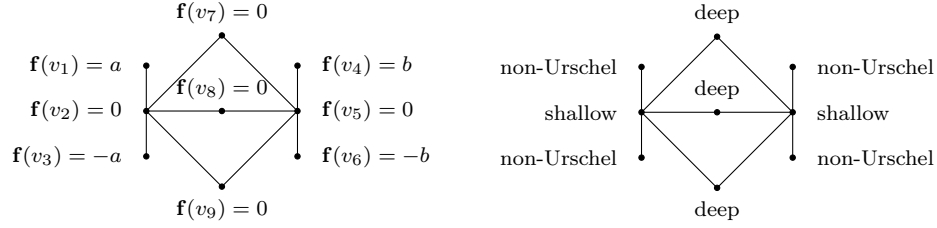

\begin{definition}
Let $\mec f\from V_G\to\reals$ be any nonzero 
function, and let $E$ be the span of $\mec f$
(therefore $\dim(E)=1$).  An {\em Urschel vertex of $\mec f$} is an
Urschel vertex of $E$, and similarly for {\em non-Urschel vertices},
{\em shallow Urschel vertices} 
and {\em deep Urschel vertices}.  We similarly write
$$
{\rm Urschel}(\mec f),\ {\rm NonUrschel}(\mec f),
\ {\rm Shallow}(\mec f),\ {\rm Deep}(\mec f).
$$
\end{definition}

See Section~\ref{se_urschel_examples} for some interesting examples of
Urschel vertices in the eigenspaces of certain Laplacian eigenvalues
of certain graphs (the reader may wish to look at these examples before
reading the rest of this section).

\subsection{Our Main Theorems}

We now state our main theorems, whose proofs are given in 
Sections~\ref{se_gen_Laplacian_perturb}
and~\ref{se_gen_Laplacian_perturb_multiple}.

The first theorem deals with a simple eigenvalue (multiplicity $1$)
whose eigenfunction has no deep Urschel vertices.

\begin{theorem}\label{th_max_urschel_number_simple_eig_all_shallow}
Let $M$ be a generalized Laplacian on a connected simple graph, $G$, and
say that $\lambda_k$ is a simple eigenvalue of $M$.
Say that $\mec f_k$ is an eigenvector of $\lambda_k$ such that
$$
{\rm Urschel}(\mec f_k) = \{ v\in V_G \ | \ \mec f_k(v)=0 \}
$$
is nonempty, but $\mec f_k$ has only shallow Urschel points.
Then for any
$$
\mec f_{k,1}\in\reals^{{\rm Urschel}(\mec f_k)}
$$
there is an $M_1$ supported in $G$ such that 
\begin{equation}\label{eq_find_any_specified_f_k_one_only_shallow_U_vert}
\mec f_{k,1}=(M-\lambda_k I)^+ M_1 \mec f_k .
\end{equation} 
In particular, we can find an $\mec f_{k,1}$ for any sign pattern specified
on the Urschel vertices.
Hence
\begin{equation}\label{eq_UN_max_f_k_at_most_k_no_deep_Urschel}
\MUN(\mec f_k)\le k.
\end{equation}
\end{theorem}

The next theorem deals with a simple eigenvalue where we allow
an arbitrary number of deep Urschel vertices.

\begin{theorem}\label{th_Urschel_reproved_with_two_signings}
Let $M$ be a generalized Laplacian on a connected simple graph, $G$, and
say that $\lambda_k$ is a simple eigenvalue of $M$.
Say that $\mec f_k$ is an eigenvector of $\lambda_k$ such that
$$
{\rm Urschel}(\mec f_k) = \{ v\in V_G \ | \ \mec f_k(v)=0 \}
$$
is nonempty.
Then there are generalized Laplacians $M'$ and $M''$ on $M$ such that:
the $k$-th eigenvalues of $M',M''$ are simple, and
there are corresponding eigenvectors $\mec f_k',\mec f_k''$ respectively such that:
\begin{enumerate}
\item 
$\mec f_k',\mec f_k''$ are each a signing of $\mec f_k$, and they are different signings
of $\mec f_k$;
\item 
$\mec f_k',\mec f_k''$ are nowhere zero, and both have $\le k$ nodal domains.
\end{enumerate}
Hence
\begin{equation}\label{eq_UN_two_f_k_at_most_k}
{\rm UN}_2(\mec f_k)\le k.
\end{equation} 
\end{theorem}

{\myred
The next theorem deals with a simple eigenvalue where we prove
${\rm UN}_{2^s}(\mec f_k)\le k$ where $s$ is the number of shallow
Urschel vertices; this generalizes both
\eqref{eq_UN_max_f_k_at_most_k_no_deep_Urschel}
and
\eqref{eq_UN_two_f_k_at_most_k}.

\begin{theorem}\label{th_Urschel_most_general}
Let $M$ be a generalized Laplacian on a connected simple graph, $G$, and
say that $\lambda_k$ is a simple eigenvalue of $M$.
Say that $\mec f_k$ is an eigenvector of $\lambda_k$ and let
$s$ be the number of shallow Urschel vertices of $\mec f_k$.
Then
$$
{\rm UN}_{2^s}(\mec f_k)\le k.
$$
\end{theorem}

This theorem shows that the less the number of deep Urschel points
an eigenfunction of a simple eigenvalue has, the closer to
the bound ${\rm UN}_{\max{}}(\mec f_k)\le k$ one gets.
Even though Theorem~\ref{th_Urschel_most_general} implies
\eqref{eq_UN_max_f_k_at_most_k_no_deep_Urschel}
and
\eqref{eq_UN_two_f_k_at_most_k},
we state Theorems~\ref{th_max_urschel_number_simple_eig_all_shallow}
and~\ref{th_Urschel_reproved_with_two_signings} beforehand because their
proofs are easier and together form a stepping stone to the proof of
Theorem~\ref{th_Urschel_most_general}.
In Section~\ref{se_shallow_deep} we give examples of graphs and generalized
Laplacians with simple eigenvalues whose eigenvector has an
arbitrary number of shallow and deep Urschel vertices and where
Theorem~\ref{th_Urschel_most_general} is, in a sense, tight.
}

Our final main result deals with multiple eigenvalues; we get a stronger
upper bound than ${\rm UN}(\mec f_j)\le j$ for some values of $j$ when we
have a multiple eigenvalue.

\begin{theorem}\label{th_perturbation_at_multiple_eigenvalue_main}
Let $M$ be a generalized Laplacian on a graph $G$ with eigenvalues
\eqref{eq_eigenvalues_increasing_order}.
Let $\lambda_k$ be an eigenvalue of multiplicity $m$ (as in
\eqref{eq_kth_eigenvalue_has_multiplicity_m_intro}).
Then there exists an orthonormal eigenbasis 
$\mec f_k,\mec f_{k+1},\ldots,\mec f_{k+m-1}$ of $E(M,\lambda_k)$
such that 
\begin{equation} 
\label{eq_UN_multiple_eigenvalue_two_sided_bound}
\mbox{\rm for all $0\le j\le m-1$},\quad
{\rm UN}(\mec f_{k+j}) \le  k+\min(j,m-1-j) .
\end{equation} 
\end{theorem}

\section{Urschel Vertices and Examples}
\label{se_urschel_examples}

In this section we give some examples that illustate our main results.

Example~\ref{ex_star_Urschel_example} discussed the star on
$n$ vertices, used in Section~\ref{se_main}
to illustrate a number of results regarding
Urschel numbers.


In Subsection~\ref{su_odd_length_paths} we use odd length paths
to illustrate
Theorem~\ref{th_max_urschel_number_simple_eig_all_shallow}, showing that
when a simple eigenvalue, $\lambda_k$, has a corresponding eigenvector, $\mec f_k$,
with no deep vertices, then
$$
\MUN(\mec f_k) \le k.
$$

In Subsection~\ref{su_simple_ladder_with_multiple_middle_rungs}
we describe examples, namely ``ladders,''
where $\Delta_G$ has eigenvalues $\lambda_3=\lambda_4=1$; here
${\rm UN}(\mec f)=3$ for all $\mec f\in E(\Delta_G,1)$ except on a set of
measure $0$, where ${\rm UN}(\mec f)=2$.  But perturbation theory finds
these exceptional $\mec f$ with ${\rm UN}(\mec f)=2$.
Ladders are also have $\MUN(\mec f)\ge n-2$ for each $\mec f\in E(\Delta_G,1)$,
although $1=\lambda_3=\lambda_4$ is a multiple eigenvalue.
$E(\Delta_G,1)$ here has $n-6$ deep Urschel vertices.

In Subsection~\ref{su_path_with_two_left_ends} we give a ``path with two
left ends''
on $n$ vertices such that $\lambda_k$ with $k$ roughly $n/3$
has $\MUN(\mec f_k)=n-1$; here $\lambda_k$ is a simple eigenvalue
for $n \bmod 3=-1,0$.  
In these cases, $E(\Delta_G,\lambda_k)$ has $n-4$ deep Urschel vertices.

The proofs of the claims regarding
Subsections~\ref{su_simple_ladder_with_multiple_middle_rungs} and
Subsection~\ref{su_path_with_two_left_ends} are given
in Appendix~\ref{ap_details_of_ladders_and_split_paths}.

\subsection{Example: Odd Length Paths}
\label{su_odd_length_paths}

For $n\in\naturals$ a {\em path on $n$ vertices}
is any graph, $G=(V_G,E_G)$ whose vertices can be arranged as
$V_G=\{v_1,\ldots,v_n\}$ such that $G$ has $n-1$ edges,
$$
E_G = \bigl\{ \{1,2\},\ \{2,3\},\ \ldots\ ,\ \{n-1,n\} \bigr\};
$$
we depict the path as
$$
\begin{tikzpicture}[scale=0.25]
\foreach \i in {1,...,5} {
  \filldraw (\i * 3,0) circle (5pt);
  \node at (\i*3,-1.5) {\Small $v_\i$};
}
\draw (3,0) to (16,0);
\node at (17.5,0) {$\cdots$};
\draw (19,0) to (23,0);
\filldraw (20,0) circle (5pt);
\filldraw (23,0) circle (5pt);
\node at (20,-1.5) {\Small $v_{n-1}$};
\node at (23,-1.5) {\Small $v_n$};
\end{tikzpicture}
$$
If we consider the generalized Laplacian $M$ given by
\begin{equation}\label{eq_generalized_Laplacian_for_path}
M(i,j) = 
\left\{ \begin{array}{ll} 
2 & \mbox{if $i=j$,} \\
-1 & \mbox{if $|i-j|=1$, and} \\
0 & \mbox{otherwise.} \\
\end{array}\right.
\end{equation} 
We caution the reader that $M$ is not the graph Laplacian, $\Delta_G$,
since $\Delta_G(1,1)=\Delta_G(n,n)=1$, not $2$.

In this case there is a standard trick to find
the eigenvalues of $M$: namely, we embed the path of $n$ vertices, $P_n$, in
$C=C_{2n+2}$, the
cycle of length $2n+2$: we number the cycle's vertices $0,1,\ldots,2n+1$,
and view the path as the induced subgraph on the vertex 
set $[n]=\{1,\ldots,n\}$.
Then we observe that if $\mec f$ is an eigenfunction of $\Delta_C$, the
Laplacian of $C=C_{2n+2}$, and if $\mec f(0)=\mec f(n+1)=0$, then $\mec f$ restricted
to $[n]$ is an eigenfunction of $M$ above, with the same eigenvalue.

It is now a standard result and easy to check 
that the $n$ eigenvalues of $\Delta_C$
with a corresponding eigenvector $\mec f$ satisfying $\mec f(0)=\mec f(n+1)=0$
are the 
eigenvalues 
$\lambda_m=2-2\cos(2\pi m/(2n+2))$ for $m\in[n]$, each with one
corresponding eigenvector $\mec f_m$ given by
$\mec f_m(j)=\sin(2\pi mj/(2n+2))$ for $j=0,1,\ldots,2n+1$.
Restricting the $\mec f_m$ to $[n]$ gives the desired eigenvectors
of $P_n$.

\begin{example}
Let $G=(V_G,E_G)$
be a path on $n$ vertices with $n$ odd and $M$ as in
\eqref{eq_generalized_Laplacian_for_path}. Then $\lambda_{(n+1)/2}=1$
has multplicity one, and an eigenvector, $\mec f$, is given as
$$
\begin{tikzpicture}[scale=0.25]
\foreach \i in {1,...,5} {
  \filldraw (\i * 3,0) circle (5pt);
  \node at (\i*3,-1.5) {\Small $v_\i$};
}
\draw (3,0) to (16,0);
\node at (17.5,0) {$\cdots$};
\draw (19,0) to (23,0);
\filldraw (20,0) circle (5pt);
\filldraw (23,0) circle (5pt);
\node at (20,-1.5) {\Small $v_{n-1}$};
\node at (23,-1.5) {\Small $v_n$};
\node at (0,1.5) {\Small $\mec f$};
\node at (3,1.5) {\Small $1$};
\node at (6,1.5) {\Small $0$};
\node at (9,1.5) {\Small $-1$};
\node at (12,1.5) {\Small $0$};
\node at (15,1.5) {\Small $1$};
\node at (20,1.5) {\Small $0$};
\node at (25,1.5) {\Small $(-1)^{(n-1)/2}$};
\end{tikzpicture}
$$
Using U and N to depict the Urschel and non-Urschel vertices, they are
$$
\begin{tikzpicture}[scale=0.25]
\foreach \i in {1,...,5} {
  \filldraw (\i * 3,0) circle (5pt);
  \node at (\i*3,-1.5) {\Small $v_\i$};
}
\draw (3,0) to (16,0);
\node at (17.5,0) {$\cdots$};
\draw (19,0) to (23,0);
\filldraw (20,0) circle (5pt);
\filldraw (23,0) circle (5pt);
\node at (20,-1.5) {\Small $v_{n-1}$};
\node at (23,-1.5) {\Small $v_n$};
\node at (0,1.5) {\Small $\mec f$};
\node at (3,1.5) {\Small $1$};
\node at (6,1.5) {\Small $0$};
\node at (9,1.5) {\Small $-1$};
\node at (12,1.5) {\Small $0$};
\node at (15,1.5) {\Small $1$};
\node at (20,1.5) {\Small $0$};
\node at (25,1.5) {\Small $(-1)^{(n-1)/2}$};
\node at (-6,3) {\Small Urschel versus non-Urschel};
\node at (3,3) {\Small N};
\node at (6,3) {\Small U};
\node at (9,3) {\Small N};
\node at (12,3) {\Small U};
\node at (15,3) {\Small N};
\node at (20,3) {\Small U};
\node at (23,3) {\Small N};
\end{tikzpicture}
$$
Hence $\mec f$ has only shallow Urschel vertices, since each Urschel vertex
is adjacent to a non-Urschel vertex.
It is easy to see that
$$
{\rm UN}(\mec f)=\MUN(\mec f)=(n+1)/2,
$$
i.e., any signing of $\mec f$ has $(n+1)/2$ nodal regions.
This illustrates
Theorem~\ref{th_max_urschel_number_simple_eig_all_shallow},
that whenever $\lambda_j$ is a simple, $j$-th eigenvalue of a generalized
Laplacian, then $\MUN(\mec f_j)\le j$ if $\mec f_j$ has no deep Urschel vertices.
\end{example}

\subsection{A Simple Ladder With
(Possibly) Multiple Rungs}
\label{su_simple_ladder_with_multiple_middle_rungs}

\begin{example}
Consider the following graph:
$$
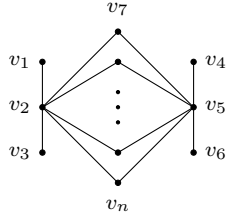
\begin{tikzpicture}[scale=0.20]
\tikzmath{ 
  int \j ;
}
\foreach \i in {1,...,3} {
  \filldraw (0,6 - \i * 3) circle (5pt);
  \node at (-1.5,6 - \i * 3) {\Small $v_\i$};
  \filldraw (10,6 - \i * 3) circle (5pt);
  \tikzmath{
    \j = \i + 3;
  }
  \node at (11.5,6 - \i * 3) {\Small $v_\j$};
  \filldraw (5, 2 - \i * 1 ) circle (3pt);
}
\draw (0,3) to (0,-3);
\draw (10,3) to (10,-3);
\filldraw (5,5) circle (5pt);
\node at (5,6.5) {\Small $v_7$};
\filldraw (5,-5) circle (5pt);
\node at (5,-6.5) {\Small $v_n$};
\draw (0,0) to (5,5);
\draw (10,0) to (5,5);
\draw (0,0) to (5,-5);
\draw (10,0) to (5,-5);
\filldraw (5,3) circle (5pt);
\filldraw (5,-3) circle (5pt);
\draw (0,0) to (5,3);
\draw (10,0) to (5,3);
\draw (0,0) to (5,-3);
\draw (10,0) to (5,-3);
\end{tikzpicture}
$$
\end{example}
This graph consists of $n\ge 7$ vertices, and is the union of
two paths of $3$ vertices (with vertices $v_1,v_2,v_3$ in one path,
and $v_4,v_5,v_6$ in the other), plus vertices $v_7,\ldots,v_n$, each
of which is joined by an edge to $v_2$ and another to $v_5$.
In particular, the case of $n=7$ is a ladder with one middle rung:
$$
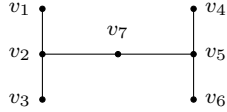
\begin{tikzpicture}[scale=0.20]
\tikzmath{ 
  int \j ;
}
\foreach \i in {1,...,3} {
  \filldraw (0,6 - \i * 3) circle (5pt);
  \node at (-1.5,6 - \i * 3) {\Small $v_\i$};
  \filldraw (10,6 - \i * 3) circle (5pt);
  \tikzmath{
    \j = \i + 3;
  }
  \node at (11.5,6 - \i * 3) {\Small $v_\j$};
}
\draw (0,3) to (0,-3);
\draw (10,3) to (10,-3);
\filldraw (5,0) circle (5pt);
\node at (5,1.5) {\Small $v_7$};
\draw (0,0) to (5,0);
\draw (10,0) to (5,0);
\end{tikzpicture}
$$
We easily see that 
$$
0 = \lambda_1 < \lambda_2 < \lambda_3=\lambda_4=1 < \lambda_5
$$
(for details, see 
Subsection~\ref{su_details_of_ladder_with_multiple_middle_rungs})
and the eigenvectors with eigenvalues $1=\lambda_3=\lambda_4$ is the
two-dimensional eigenspace
of functions $\mec f\from V\to\reals$ given by:
\begin{equation}
\label{eq_eigenvalue_one_for_ladder_with_multiple_rungs}
\begin{tikzpicture}[scale=0.20]
\node at (-22,0) {$\forall a,b\in\reals$,};
\tikzmath{ 
  int \j ;
}
\foreach \i in {1,...,3} {
  \filldraw (0,6 - \i * 3) circle (5pt);
  \filldraw (10,6 - \i * 3) circle (5pt);
  \tikzmath{
    \j = \i + 3;
  }
  \filldraw (5, 2 - \i * 1 ) circle (3pt);
}
\node[anchor=east]  at (-1,3) {\Small $\mec f(v_1)=a$};
\node[anchor=east] at (-1,0) {\Small $\mec f(v_2)=0$};
\node[anchor=east] at (-1,-3) {\Small $\mec f(v_3)=-a$};
\node[anchor=west] at (11,3) {\Small $\mec f(v_4)=b$};
\node[anchor=west] at (11,0) {\Small $\mec f(v_5)=0$};
\node[anchor=west] at (11,-3) {\Small $\mec f(v_6)=-b$};
\draw (0,3) to (0,-3);
\draw (10,3) to (10,-3);
\filldraw (5,5) circle (5pt);
\node at (5,6.5) {\Small $\mec f(v_7)=0$};
\filldraw (5,-5) circle (5pt);
\node at (5,-6.5) {\Small $\mec f(v_n)=0$};
\draw (0,0) to (5,5);
\draw (10,0) to (5,5);
\draw (0,0) to (5,-5);
\draw (10,0) to (5,-5);
\filldraw (5,3) circle (5pt);
\filldraw (5,-3) circle (5pt);
\draw (0,0) to (5,3);
\draw (10,0) to (5,3);
\draw (0,0) to (5,-3);
\draw (10,0) to (5,-3);
\end{tikzpicture}
\end{equation} 

It follows that for $\lambda_3=\lambda_4=1$, $v_7,\ldots,v_n$ are
deep Urschel vertices, and $v_2,v_5$ are shallow Urschel vertices.

Note that for $a,b$ both nonzero,
$\mec f$ in 
\eqref{eq_eigenvalue_one_for_ladder_with_multiple_rungs}
has ${\rm UN}(\mec f)={\rm UN}_2(\mec f)=3$ (by taking all vertices
$v_2,v_5$ and $v_7,\ldots,v_n$ to have the same sign) and
${\rm UN}_3(\mec f)=4$; also
$\MUN(\mec f)=n-2$ by taking
$v_2,v_5$ are signed positively, and $v_7,\ldots,v_n$ negatively.

Notice, however, that the function in
\eqref{eq_eigenvalue_one_for_ladder_with_multiple_rungs} with
$a\ne 0$ and $b=0$ has Urschel number $2$, and the same with
$a=0$ and $b\ne 0$.  It turns out that perturbation theory
will discover an eigenbasis of these two functions.

\subsection{Paths With Two Left Ends}
\label{su_path_with_two_left_ends}

Here is another class of graphs with deep Urschel points; this class
includes a graph with 5 vertices.  We depict these graphs as:
$$
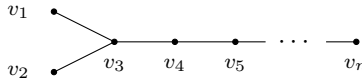
\begin{tikzpicture}[scale=0.20]
\foreach \i in {3,...,5} {
  \filldraw (4*\i-8,0) circle (5pt);
  \node at (4*\i-8,-1.5) {\Small $v_\i$};
}
\filldraw (0,2) circle (5pt);
\filldraw (0,-2) circle (5pt);
\node [anchor=east] at (-1,2) {\Small $v_1$};
\node [anchor=east] at (-1,-2) {\Small $v_2$};
\draw (0,2) to (4,0);
\draw (0,-2) to (4,0);
\draw (4,0) to (14,0);
\node at (16,0) {$\cdots$};
\node at (20,-1.5) {\Small $v_n$};
\draw (18,0) to (20,0);
\filldraw (20,0) circle (5pt);
\end{tikzpicture}
$$
Hence this is a path of $n-1$ vertices,
i.e., with vertices $v_1,v_3,v_4,\ldots,v_n$,
plus a vertex $v_2$ that is incident upon a single edge $\{v_2,v_3\}$.
We shall call this graph a {\em an $n$-vertex path with a double left end}.

It is not hard to see that:
\begin{enumerate}
\item
For $n=3m-1$ and $n=3m$ with
$m\in\naturals$, the eigenvalue $1=\lambda_{m+1}$
has multiplicity $1$,
with eigenvector $\mec f$ given by
\begin{equation}\label{eq_deep_urschel_points_in_two_left_ended_path}
\begin{tikzpicture}[scale=0.20]
\node at (-15,0) {$\forall a\in \reals$};
\foreach \i in {3,...,5} {
  \filldraw (4*\i-8,0) circle (5pt);
}
\node [anchor=west] at (3,-2) {\Small $\mec f(v_3)=\mec f(v_4)=\cdots=\mec f(v_n)=0$};
\filldraw (0,2) circle (5pt);
\filldraw (0,-2) circle (5pt);
\node [anchor=east] at (-1,2) {\Small $\mec f(v_1)=a$};
\node [anchor=east] at (-1,-2) {\Small $\mec f(v_2)=-a$};
\draw (0,2) to (4,0);
\draw (0,-2) to (4,0);
\draw (4,0) to (14,0);
\node at (16,0) {$\cdots$};
\draw (18,0) to (20,0);
\filldraw (20,0) circle (5pt);
\end{tikzpicture}
\end{equation} 
Hence $v_3,v_4,\ldots,v_n$ are deep Urschel vertices.
By assigning $v_3,v_4,\ldots,v_n$ to have alternating signs,
one can get $n-1$ nodal domains.
\item
Curiously, for $n=3m+1$ and $m\in\naturals$, 
the eigenvalue $1=\lambda_{m+1}=\lambda_{m+2}$ has multiplicity $2$,
with the above eigenvector plus another
$$
\begin{tikzpicture}[scale=0.20]
\node at (-15,0) {$\forall a\in \reals$};
\foreach \i in {3,...,9} {
  \filldraw (4*\i-8,0) circle (5pt);
}
\filldraw (0,2) circle (5pt);
\filldraw (0,-2) circle (5pt);
\node [anchor=east] at (-1,2) {\Small $\mec f(v_1)=a$};
\node [anchor=east] at (-1,-2) {\Small $\mec f(v_2)=a$};
\node at (4,-2) {\Small $0$};
\node at (8,-2) {\Small $-2a$};
\node at (12,-2) {\Small $-2a$};
\node at (16,-2) {\Small $0$};
\node at (20,-2) {\Small $2a$};
\node at (24,-2) {\Small $2a$};
\node at (28,-2) {\Small $0$};
\draw (0,2) to (4,0);
\draw (0,-2) to (4,0);
\draw (4,0) to (30,0);
\node at (32,0) {$\cdots$};
\draw (4,0) to (30,0);
\draw (34,0) to (40,0);
\filldraw (36,0) circle (5pt);
\filldraw (40,0) circle (5pt);
\node at (36,-2) {\Small $0$};
\node at (42,-2) {\Small $2a(-1)^m$};
\end{tikzpicture}
$$
Hence, in this case the Urschel vertices are $v_3,v_6,\ldots,v_{3m}$
and there are no deep Urschel vertices.
\end{enumerate}
A detailed proof of these spectral results will be given
in Subsection~\ref{su_path_with_two_left_ends_details}.

\section{Perturbations Theory}
\label{se_perturbation}

In this section we review the perturbation theory we need, and prove an
important lemma (Lemma~\ref{le_perturb_subspace_E_at_a_diagonal_matrix})
that we will need in Section~\ref{se_gen_Laplacian_perturb_multiple}.

The following is standard perturbation theory; we will refer to
parts of \cite{kato}; there are simpler texts for the case of a simple
eigenvalue, but we will need \cite{kato} when it comes to multiple
eigenvalues.

Throughout most of this section
we work with arbitrary symmetric matrices $M_0,M_1\in \reals^{n\times n}$.
(In \cite{kato}, the matrices could be, more generally, 
complex Hermitian matrices.)

\subsection{Motivation for Using Perturbation Theory}

The key to relating perturbation theory to Urschel nodal domains
is the following simple proposition, whose proof is immediate.

\begin{proposition}
Let $G$ be a graph, and 
$\mec f^0,\mec f^1\in\reals^{V_G}$, and let $\mec f(\epsilon)$
be any family of elements of $\reals^{V_G}$ depending on a parameter
$\epsilon\in\reals$ such that for small $|\epsilon|$ we have
$$
\mec f(\epsilon)=\mec f^0+\epsilon\, \mec f^1+O(\epsilon^2),
$$
i.e.,
$$
\forall v\in V_G,\quad
\mec f(\epsilon)(v)=\mec f^0(v)+\epsilon \mec f^1(v) + O(\epsilon^2).
$$
Then,
\begin{enumerate}
\item
if $v\in{\rm NonUrschel}(\mec f^0)$, then for $|\epsilon|$ sufficiently
small, $v$ is a non-Urschel point of
$\mec f(\epsilon)$ with the same sign as $\mec f^0$ at $v$;
\item
if $v\in{\rm Urschel}(\mec f^0)$ and $f^1(v)\ne 0$, 
then for $|\epsilon|$ sufficiently
small and nonzero, $v$ is a non-Urschel point of 
{\myred $\mec f(\epsilon)$} with:
\begin{enumerate}
\item the same sign as $f^1(v)$ for $\epsilon>0$; and
\item the opposite sign as $f^1(v)$ for $\epsilon<0$.
\end{enumerate}
\end{enumerate}
\end{proposition}
\begin{proof}
If $a,b\in\reals$, then if $a\ne 0$, then $a+\epsilon b$ has the same sign
as $a$ for $\epsilon$ sufficiently small; if $a=0$, then $a+\epsilon b$
has the same sign as $\epsilon b$.
We now apply this to $a=f^0(v)$ and $b=f^1(v)$ at each (of finitely many)
$v\in V_G$,
{\myred and note that the sign of $a+\epsilon b$ and 
$a+\epsilon b+O(\epsilon^2)$ are the same for small $|\epsilon|>0$
unless $a=b=0$.}
\end{proof}

\subsection{Perturbation Theory of Symmetric Matrices at a Simple or
Multiple Eigenvalue}
\label{su_general_perturbation_theory}

Let $M_0,M_1$ be two $n\times n$ symmetric matrices, and
for $\epsilon\in\reals$,
$$
M(\epsilon)=M_0+\epsilon M_1.
$$
Kato \cite{kato}, Section~II.6.1
refers to $M(\epsilon)$ with $\epsilon$ varying over 
the complex numbers as {\em symmetric} (since $M_0,M_1$ are symmetric).
In the case, Theorem~II.6.1 of \cite{kato} states that 
$M(\epsilon)$ has eigenvalues 
$\lambda_1(\epsilon),\ldots\lambda_N(\epsilon)$, which are distinct
holomorphic functions, 
and each eigenprojection, $P_h(\epsilon)$, i.e., the projection onto
$E(M(\epsilon),\lambda_h(\epsilon))$, for $h\in [N]$,
is holomorphic in $\epsilon$.
Moreover (first paragraph, Section~II.6.2) there are
holomorphic eigenfuctions $\phi_1(\epsilon),\ldots,\phi_n(\epsilon)$
for $M(\epsilon)$ (so if $N<n$, those eigenspaces of multiple eigenvalues,
$E(M(\epsilon),\lambda_h(\epsilon))$, are spanned by more than one
of the $\phi_1(\epsilon),\ldots,\phi_n(\epsilon)$.

\subsection{Perturbation Theory at a Simple Eigenvalue}

In this subsection we describe perturbation theory at a simple
eigenvalue.  What we claim the eigenvalue perturbation formula
\eqref{eq_lambda_one_as_ratio_involving_f_zero_and_M_one}
in this subsection is really a special case
of Theorem~\ref{th_perturbation_theory_at_a_multiple_eigenvalue}.
However, the analysis in this section is much simpler
and provides important intuition
for Theorem~\ref{th_perturbation_theory_at_a_multiple_eigenvalue}.

Say that $M_0,M_1\in\reals^{n\times n}$ are symmetric matrices,
and let
$$
M(\epsilon) = M_0 + \epsilon M_1.
$$
Say that $\lambda_0$ is a simple eigenvalue
of $M_0$, and let
$\mec f^0\in\reals^n\setminus\{\mec 0\}$ be a corresponding eigenfunction,
i.e., 
$M_0 \mec f^0 = \lambda_0 \mec f^0$.
Then by the general theory
(Subsection~\ref{su_general_perturbation_theory}),
for real $\epsilon$ sufficiently near $0$ there
are convergent power series
\begin{align}
\label{eq_f_expansion_at_espilon_near_zero}
\mec f(\epsilon)= \mec f^0 + \epsilon \mec f^1 + \epsilon^2 \mec f^2 + \cdots ,
\\
\label{eq_lambda_expansion_at_espilon_near_zero}
\lambda(\epsilon)= \lambda_0 + \epsilon \lambda_1 + \epsilon^2 \lambda_2
+ \cdots
\end{align} 
such that
\begin{equation}\label{eq_eigenpair_equation_depending_on_epsilon}
M(\epsilon)\mec f(\epsilon)=\lambda(\epsilon)\mec f(\epsilon) .
\end{equation} 
The order $\epsilon^1$ term of \eqref{eq_eigenpair_equation_depending_on_epsilon}
reads:
\begin{equation}\label{eq_epsilon_term_of_M_epsilon_eigenpair}
M_0 \mec f^1 + M_1 \mec f^0  = \lambda_0 \mec f^1 + \lambda_1 \mec f^0.
\end{equation} 
Of course, we may always multiply $\mec f(\epsilon)$ by a scalar power
series $p(\epsilon)=1+\epsilon p_1 + \epsilon^2 p_2 + \cdots$,
so $\mec f(\epsilon)$ is
not uniquely determined; however, the effect of this is to replace
$\mec f^1$ by $\mec f^1 + p_1 \mec f^0$, and hence
$\mec f^1$ is uniquely determined if we insist that $\mec f^1$ is orthogonal
to $\mec f^0$;
assume so.
Taking the dot product of \eqref{eq_epsilon_term_of_M_epsilon_eigenpair}
with $\mec f_0$ therefore yields
$$
(\mec f^0,M_0 \mec f^1) + (\mec f^0, M_1 \mec f^0)  
= \lambda_0 (\mec f^0,\mec f^1) + \lambda_1 (\mec f^0,\mec f^0),
$$
and therfore (using $(\mec f^0,M_0 \mec f^1)=(M_0 \mec f^0,\mec f^1)=(\lambda_0 \mec f^0,\mec f^1)$)
we get
\begin{equation}\label{eq_lambda_one_as_ratio_involving_f_zero_and_M_one}
\lambda_1 = \frac{(\mec f^0,M_1 \mec f^0)}{(\mec f^0,\mec f^0)}.
\end{equation} 

(Notice that this is the Rayleigh quotient of $M_1$ applied to $\mec f^0$,
and all the above are well-known; for a formal proof, note that all the
above can be derived from
Theorem~\ref{th_perturbation_theory_at_a_multiple_eigenvalue} below,
whose proof we will give from various parts of
\cite{kato}.)

This also gives
$$
(M_0-\lambda_0 I) \mec f^1 = (\lambda_1 I - M_1)\mec f^0,
$$
and therefore 
\begin{equation}\label{eq_f_one_as_pseudoinverse_with_M_one_f_zero}
\mec f^1 = (M_0-\lambda_0 I)^+ (\lambda_1 I - M_1)\mec f^0 
=
(M_0-\lambda_0 I)^+ (- M_1 \mec f^0),
\end{equation} 
where $(M_0-\lambda_0 I)^+$ is the pseudo-inverse of $M_0-\lambda_0 I$.

\subsection{Perturbation at a Multiple Eigenvalue}

The perturbation theory at a multiple eigenvalue is more subtle.

\begin{theorem}\label{th_perturbation_theory_at_a_multiple_eigenvalue}
Say that $M_0,M_1$ are $n\times n$ symmetric, and $\lambda_0$ is an
eigenvalue of $M_0$ of multiplicity $m\ge 1$, and let $E(M_0,\lambda_0)$
be the eigenspace corresponding to $\lambda_0$;
let $P=P^{\rm T}=P^2$ be the orthogonal projection onto $E(M_0,\lambda_0)$.
Say that $PM_1P\from E\to E$ (which is a symmetric operator on $E$)
has distinct eigenvalues
$$
\lambda_1'<\lambda_2'<\cdots<\lambda_m'.
$$
Then for small nonzero $|\epsilon|$, $M(\epsilon)=M_0+\epsilon M_1$
has distict eigenvalues
$$
\lambda_j(\epsilon)=\lambda_0 + \epsilon \lambda_j' + O(\epsilon^2)
$$
and corresponding to $\lambda_j(\epsilon)$ is an eigenvector
$$
\mec f_j(\epsilon)=\mec f_j^0+\epsilon \mec f_j^1 + O(\epsilon),
$$
where $\mec f_j^0,\mec f_j^1\in\reals^n$ and
$\mec f_1^0,\ldots,\mec f_m^0$ are mutually orthogonal.
\end{theorem}

For a proof of this theorem, see
Theorem~6.8, Section~II.6.3 (page~141) of \cite{kato} for the
statement regarding $\lambda_j(\epsilon)$;\footnote{
  It is a bit hard to extract this theorem from Kato's textbook
  \cite{kato}, so let us give the translation of terms:
  in \cite{kato}, $T(x)$ is symmetric and continuously differentiable
  in $x$; hence $T(0)$ is our $M_0$, and $T'(0)$ is our $M_1$, and
  $T(x)=M_0+xM_1$.  Kato's $PT'(0)P$ is our $PM_1P$, and Kato's
  $\textsf{M}$ is our $E$.
  We are assuming that the eigenvalues of $PM_1P$ on $E$ are distinct,
  in which case the $m$ eigenvalues are $\lambda'_1,\ldots,\lambda'_m$
  which is Kato's notation in the special case $p=m$ there.
  So each of Kato's ``$\lambda+x\lambda_j'$-group'' (see page~93 of
  \cite{kato} for this definition) refers to a simple
  eigenvalue given as 
  $\lambda_j(\epsilon)=\lambda+\epsilon\lambda_j'+o(\epsilon)$.
  Since we are assuming $T(x)=M_0+xM_1$, which is infinitely differentiable
  in $x$, we can replace the $o(\epsilon)$ with $O(\epsilon^2)$; but the
  $o(\epsilon)$ would be just as good for us in our applications.
  }
see also 
the beginning of Subsection~\ref{su_general_perturbation_theory},
where we quote Sections~II.6.1 and~II.6.2 of \cite{kato}, which shows 
that $\mec f_1(\epsilon),\ldots,\mec f_m(\epsilon)$ are orthonormal for all
$\epsilon$, and therefore are orthonormal for $\epsilon=0$, and therefore
$\mec f_1^0,\ldots,\mec f_m^0$ are orthonormal.\footnote{
  We remark that Section~II.6.3 of \cite{kato} assumes only that
  $M(\epsilon)$ is continuously differentiable in $\epsilon$; but in this
  case, Example~II.5.3 (page~128), shows that even if $M(\epsilon)$ is
  infinitely differentiable in $\epsilon$, but not holomorphic in
  $\epsilon$, then the eigenprojections cannot necessarily be extended
  continuously in a neighbourhood of $\epsilon=0$.
  }

[Intuitively, 
the theorem above generalizes the formula for a simple
eigenvalue $\lambda_0$ whose perturbed eigenvalue
$\lambda_0+\epsilon\lambda_1 + O(\epsilon^2)$ is given by
$$
\lambda_1 = \cR_{M_1}(\mec f^0) 
\quad\mbox{where}\quad
\cR_{M_1}(\mec f)\eqdef \frac{(\mec f,M_1\mec f)}{(\mec f,\mec f)}.
$$
In the case of a multiple eigenvalue, we consider the critical
points of $\cR_{M_1}(\mec f)$ where $\mec f$ ranges over $E(M_0,\lambda_0)$.
However, $M_1$ is not a map from $E(M_0,\lambda_0)$ to itself, but
$PM_1$ is, and 
we have $(\mec f,M_1\mec f)=(\mec f,PM_1\mec f)$ for all 
$\mec f\in E(M_0,\lambda_0)$.
Hence the critical points of $\cR_{M_1}(\mec f)$
are the same as the eigenvectors of
$PM_1\from E(M_0,\lambda_0)\to E(M_0,\lambda_0)$;
equivalently, these are the eigenvectors of $PM_1P$ (a map
$\reals^n\to\reals^n$) that lie in $E(M_0,\lambda_0)$.]

Hence the perturbation theory at a multiple eigenvalue $\lambda_0$ of
$M_0$ is much like that of a simple eigenvalue, provided that
$PM_1$ has only simple eigenvalues.
Otherwise, the multiple eigenvalue can persist, or else some or all of the
eigenvalues can separate, but this separation of eigenvalues
requires one to look at the
order $\epsilon^k$ coefficient for $k\ge 2$.



To perturb around a multiple eigenvalue of a generalized Laplacian, we
will need the following technical lemma.

\begin{lemma}\label{le_perturb_subspace_E_at_a_diagonal_matrix}
Let $E\subset \reals^n$ be a subspace of dimension $m$, and let
$P\from\reals^n\to E$ be the orthogonal projection.
Then there exists a diagonal matrix, $D$, such that 
$PD\from E\to E$ has $m$ distinct nonzero eigenvalues
(or, equivalently, the map $PDP\from\reals^n\to\reals^n$ have
$m$ distinct nonzero eigenvalues).
\end{lemma}

The proof is a bit technical, and will be given in 
Subsection~\ref{su_proof_of_perturb_subspace_E_at_a_diagonal_matrix}.

When we apply Lemma~\ref{le_perturb_subspace_E_at_a_diagonal_matrix},
$E$ will be the eigenspace of a multiple eigenvalue.

\begin{example}
Let $E$ be the span of $\mec e_1,\ldots,\mec e_m$.
Then $D$'s values only matter at its first $m$ diagonal entries,
and the eigenvectors of $PD$ with $m$ distinct eigenvalues are
necessarily $\mec e_1,\ldots,\mec e_m$.
Hence there is an subspace $E$ in 
Lemma~\ref{le_perturb_subspace_E_at_a_diagonal_matrix} 
where the set of eigenvectors is independent of $D$.
For this reason, the perturbation method
does not generally produce a codimension zero subset of elements, $\mec f_k$, in
$E(M,\lambda_k)$ with ${\rm UN}(\mec f_k)\le k$ (unlike Urschel's method);
an example of this is the ladder
(see Section~\ref{se_ladder_again}).
\end{example}

\subsection{Proof of Lemma~\ref{le_perturb_subspace_E_at_a_diagonal_matrix}}
\label{su_proof_of_perturb_subspace_E_at_a_diagonal_matrix}

\begin{lemma}\label{le_reduce}
Let $E\subset\reals^n$ be an $m$-dimensional subspace. 
Let $\mec v_1,\ldots,\mec v_m$ be an orthonormal basis for $E$,
and let
$Q$ be the matrix whose columns are $\mec v_1,\ldots,\mec v_m$.
Then for any symmetric matrix $M_1$, the eigenvalues of $PM_1P$
restricted to $E$ are the same as the eigenvalues of the
$m\times m$ matrix $Q^{\rm T}M_1 Q$ on $\reals^m$.
\end{lemma}
\begin{proof}
Note that $P=QQ^{\rm T}$ since $QQ^{\rm T}$ fixes
each $\mec v_i$ and $Q^{\rm T}\mec w=0$ if $w\in E^\perp$.
Since $PM_1P$ is a symmetric operator when restricted to $E$
(since $PM_1P$ is symmetric on $\reals^n$ and
$P$ takes $E^\perp$ to $0$), there is
an orthonormal basis $\mec u_1,\ldots,\mec u_m$ of $E$
of eigenvectors, namely
$PM_1P\mec u_i=\lambda_i\mec u_i$ (so $\lambda_1,\ldots,\lambda_m$ are
the eigenvalues of $PM_1P$ restricted of $E$).
Note that $Q^T$ is a map $\reals^n\to\reals^m$ that takes each element
of $E^\perp$ to $0$, and is an isomorphism $E\to\reals^m$ when restricted
to $E$ (taking each $\mec u_i$ to the $i$-th standard basis vector).
Moreover, $Q^{\rm T}$ preserves the dot product on $E$, since
it takes the orthonormal basis $\mec u_1,\ldots,\mec u_m$ to
the standard basis vectors of $\reals^m$ that are an orthonormal
basis for $\reals^m$.

Now, if 
$PM_1P\mec u_i=\lambda_i\mec u_i$, then
$$
QQ^{\rm T} M_1 QQ^{\rm T} \mec u_i = \lambda_i \mec u_i,
$$
and since the RHS lies in $E$, we can apply the isometry $Q^{\rm T}$
to the left of both sides and we find
$$
\lambda_i Q^{\rm T}\mec u_i = 
Q^{\rm T}QQ^{\rm T} M_1 QQ^{\rm T} \mec u_i 
=
Q^{\rm T} M_1 QQ^{\rm T} \mec u_i
$$
since we easily see that $Q^{\rm T}Q=I_m$ is the identity on $\reals^m$.
Hence $\mec w_i=Q^{\rm T}\mec u_i$ satisfies
$Q^{\rm T}M_1 Q \mec w_i=\lambda_i \mec w_i$,
and by isomotry, $\mec w_1,\ldots,\mec w_m$ is an orthonormal basis
for $\reals^m$ with the same eigenvalues on $E$ as $PM_1P$.
\end{proof}

Here is the main lemma.

\begin{lemma}\label{le_distinct_eigenvalues}
Let $\mec w_1,\ldots,\mec w_m\in\reals^m$ be any linearly independent
vectors.  Then for some $d_1,\ldots,d_m$, the matrix
$$
W(d_1,\ldots,d_m)
=\sum_{i=1}^m d_i \mec w_i \mec w_i^{\rm T}
$$
(which is clearly symmetric) has distinct eigenvalues.
More precisely, there is a $C>0$ such that this is always the case
provided that $d_1>0$ and $d_{i+1}> C d_i$ for $i\in[m]$.
\end{lemma}
\begin{proof}
We will restrict ourselves to choosing $d_1\le d_2\le\cdots\le d_m$.
Let $\mec y_1,\ldots,\mec y_m$ be the dual basis, i.e.,
$\mec y_i\cdot\mec w_j=\delta_{ij}$ (the Dirac delta).
If $\lambda_k$ is the $k$-th smallest eigenvalue of $W=W(d_1,\ldots,d_m)$,
then the min-max principle implies that
for any $m-k+1$ dimensional subspace $V\in\reals^m$, we have
$$
\lambda_k \le \max_{\mec v\in V} \cR_W(\mec v),
$$
where $\cR_W=\cR_{W(d_1,\ldots,d_m)}$ is the Rayleigh quotient
$$
\cR_W(\mec v) = \frac{\mec v\cdot (W\mec v)}{\mec v\cdot\mec v}.
$$

The point of choosing the dual basis is that for any $b_1,\ldots,b_m$
we easily check that
$$
\left( \sum_{i=1}^m b_i \mec y_i \right)\cdot 
\left( \sum_{i=1}^m b_i W(d_1,\ldots,d_m) \mec y_i\right)
= \sum_{i=1}^m b_i^2 d_i,
$$
which is the numerator in the Rayleigh quotient applied to
$\mec v=\sum_{i=1}^m b_i \mec y_i$.

Now choose $W$ to be the span of $\mec y_1,\ldots,\mec y_{m-k+1}$.
Let $\mec v=b_1\mec y_1+\cdots + b_{m-k+1}\mec y_{m-k+1}$
be a unit vector at which $\cR_W(\mec v)$ attains its maximum.
Then, using the fact that the $\{\mec y_i\}_{i\in[m]}$ are a dual basis,
we have
$$
\cR_W(\mec v)= b_1^2 d_1 + \cdots + b_{m-k+1}^2 d_{m-k+1}
\le d_{m-k+1} (b_1^2 + \cdots + b_{m-k+1}^2 )
\le d_{m-k+1} C_1,
$$
where 
$$
C_1 = \max \bigl\{ b_1^2 + \cdots + b_m^2 \ \bigm| 
\mbox{$\sum_{i=1}^m b_i\mec y_i$ is a unit vector} \bigr\},
$$
and $C_1$ is finite by the equivalence of norms.

Similarly, let
$$
C_2 = \min \bigl\{ b_1^2 + \cdots + b_m^2 \ \bigm|
\mbox{$\sum_{i=1}^m b_i\mec y_i$ is a unit vector} \bigr\},
$$
hence $0<C_2\le C_1$.  By the max-min principle, we have that
for any $W\subset\reals^n$ of dimension $k$ we have
$$
\lambda_k \ge \min_{v\in V} \cR_W(\mec v).
$$
So let $V$ be the span of $\mec y_{m-k+1},\ldots,\mec y_m$.
Then for any unit vector $\mec v\in V$ we have
$\mec v=b_1\mec y_m+\cdots+b_k\mec y_{m-k+1}$
and
$$
\mec v\cdot (W(d_1,\ldots,d_m)\mec v) =
b_1^2 d_m + \cdots + b_k^2 d_{m-k+1} \ge
(b_1^2+\cdots + b_k^2)d_{m-k+1}\ge C_2 d_{m-k+1}.
$$
Hence
$$
\lambda_k \ge d_{m-k+1} C_2 ,
$$
and so
$$
C_2 d_{m-k+1} \le \lambda_k \le C_1 d_{m-k+1}
$$
provided that $d_1\le d_2\le\cdots\le d_m$.  So we have that
if $C_1 d_{m-k} < C_2 d_{m-k+1}$, then
$\lambda_{k+1}<\lambda_k$.
Hence $\lambda_1<\lambda_2<\cdots<\lambda_m$ provided that
$d_i < (C_2/C_1) d_{i+1}$ for all $i\in[m-1]$
and $d_1>0$ (we also assumed that $d_1\le d_2\le\cdots\le d_m$, but
this is guaranteed by $$d_i < (C_2/C_1) d_{i+1}$$ since $C_2\le C_1$).
This proves the lemma with $C=C_1/C_2$.
\end{proof}

\begin{proof}[Proof of Lemma~\ref{le_perturb_subspace_E_at_a_diagonal_matrix}]
Apply Lemma~\ref{le_reduce}, with
$\mec v_1,\ldots,\mec v_m$ and
$Q$ as in the lemma; it suffices to analyze the eigenvalues
of $Q^T M_1 Q$ for $M_1={\rm diag}(d_1,\ldots,d_n)$.
{\myred
The matrix whose rows are $\mec v_1,\ldots,\mec v_m$ has an
$m\times m$ submatrix of full rank, and by 
rearranging the components of $\mec v_1,\ldots,\mec v_m$, we can
assume that this submatrix is the one corresponding to the first $m$ components 
of $\mec v_1,\ldots,\mec v_m$.
So for $i\in[m]$, 
}
let $\mec w_i$ be the first $m$ components of $\mec v_i$;
by assumption, $\mec w_1,\ldots,\mec w_m$ are linearly independent.
Now choosing $M_1={\rm diag}(d_1,\ldots,d_m,0,\ldots,0)$, we have
$$
W(d_1,\ldots,d_m)
\eqdef  Q^{\rm T}M_1 Q = \sum_{i=1}^m d_i \mec w_i \mec w_i^{\rm T}.
$$
Then by applying
Lemma~\ref{le_distinct_eigenvalues}
we see there exist $d_1,\ldots,d_m$ such that 
$W(d_1,\ldots,d_m)$ have distinct eigenvalues.
\end{proof}

\section{Perturbation of Generalized Laplacians at a Simple Eigenvalue}
\label{se_gen_Laplacian_perturb}

In this section we describe perturbations of generalized Laplacians.
We focus on a theorem regarding a simple eigenvalue of a generalized
Laplacian.

\begin{definition}
Let $G$ be a simple graph.  We
say that $\mec f,\tilde{\mec f}\in\reals^{V_G}$ are {\em sign equivalent}
if they have the same sign --- both positive, or both negative, or both
zero --- at each vertex.
\end{definition}
Note that if $\mec f,\tilde{\mec f}$ are sign equivalent, then
${\rm WND}(\mec f)={\rm WND}(\tilde{\mec f})$, and similarly for 
SND, UN, $\MUN{}$, and ${\rm UN}_i$.

\begin{definition}
Let $G$ be a simple graph, and $M_0$ a generalized Laplacian on
$G$.  
We say that $M_1\in \reals^{V_G\times V_G}$ is {\em supported on $G$}
if $M_1(u,v)=0$ whenever $u\ne v$ and $\{u,v\}\notin E_G$.
By a {\em perturbation of $M_0$} we mean a 
family of matrices in $\reals^{V_G\times V_G}$,
$$
M(\epsilon)=M_0+\epsilon M_1
$$
for $\epsilon\in\reals$
such that $M(\epsilon)$ is a generalized Laplacian for $|\epsilon|$
sufficiently small, or, equivalently, $M_1$ is supported on $G$.
At times we refer also to $M_1$ as a {\em perturbation of $M_0$} if confusion
is unlikely to arise.
\end{definition}

\begin{lemma}\label{le_perturbation_at_simple_eigenvalue_main}
Let $G$ be a finite, connected, simple graph, and $M_0$
any generalized Laplacian
on $G$.
Say its $k$-th smallest eigenvalue, $\lambda_k=\lambda_k(M)$ 
occurs with multiplicity one, and $\mec f_k$ is a corresponding eigenvector.
Say that $\mec f_k$ has at least one Urschel vertex.
Then there exists a generalized Laplacian perturbation of $M_0$,
$M(\epsilon)=M_0+\epsilon M_1$, such that $\mec f_k^1$ given by
\begin{equation}\label{eq_f_sub_k_sup_one_with_pseudoinverse}
\mec f_k^1 
=
(M_0-\lambda_0 I)^+ (- M_1 \mec f_k),
\end{equation} 
(compare with \eqref{eq_f_one_as_pseudoinverse_with_M_one_f_zero})
has $f_1(v)\ne 0$ for at least one Urschel vertex of $v$.
Moreover, if $\mec f_k$ 
has no deep Urschel vertices, then
we can choose $M_1$ so that
$\mec f_k^1$ 
in \eqref{eq_f_sub_k_sup_one_with_pseudoinverse}
can be taken to have any specified values on the Urschel vertices
of $\mec f_k$
(and to vanish on all non-Urschel vertices).
{\myred Finally, if $\mec f_k$ has $s$ shallow Urschel vertices, then
there is a set $V'\subset{\rm Urschel}(\mec f_k)$ with $|V'|\ge s$
such that as $M_1$ varies over all symmetric matrices supported on $G$,
$\mec f_k^1$ 
in \eqref{eq_f_sub_k_sup_one_with_pseudoinverse}
can be taken to have any values on $V'$.}
\end{lemma}

{\myred 
We depict the variables and notation used in the proof below in
Table~\ref{ta_table_of_variables_in_the_proof}.
\begin{table}[h]
\myred\centering
\begin{tabular}{|c|c|c|c|}
\hline
& ${\rm Deep}(\mec f_k)$ & ${\rm Shallow}(\mec f_k)$ 
& ${\rm NonUrschel}(\mec f_k)$ 
\\
\hline
$\mec f_k$ & zero & zero & non-zero \\
\hline
$\mec f_k^\perp$ & any & any & $U_1$ (codim $1$)\\
\hline
$W=\{M_1 \mec f_k\} $ & zero & any & any \\
\hline
$W_1=\mec f_k^\perp\cap W $ & zero & any & $U_1$ \\
\hline
\end{tabular}
\vskip 0.2 truein
\caption{The variables in the proof: $\mec f_k$ is an eigenvector,
which vanishes on
${\rm Urschel}(\mec f_k)={\rm Shallow}(\mec f_k)\cup{\rm Deep}(\mec f_k)$.
Hence $\mec f_k^\perp$ takes any values on the ${\rm Urschel}(\mec f_k)$
and is a codim 1 subspace of the functions supported 
on ${\rm NonUrschel}(\mec f_k)$.
As $M_1$ varies over perturbations of $M_0$, $M_1\mec f_k$ takes any values
on the shallow Urschel and non-Urschel vertices.
We know that the operator $(M_0-\lambda_k I)^+$ is a bijection on 
$\mec f_k^\perp$
(but we have little other information about $(M_0-\lambda_k I)^+$).
Hence $(M_0-\lambda_k I)^+W$ has dimension one less than the number
of shallow Urschel and non-Urschel vertices of $\mec f_k$.
}
\label{ta_table_of_variables_in_the_proof}
\end{table}
}
\begin{proof}
{\myred
Since $G$ is connected,
$$
s \eqdef \bigl| {\rm Shallow}(\mec f_k) \bigr| \ge 1.
$$
Recall the notation \eqref{eq_reals_B_subset_A}, so that
}
for $A\subset V_G$, we set
$$
\reals^{A\subset V_G}
= {\rm Span}\{ \mec e_a \ | \ a\in A \},
$$
where $\mec e_a$ is the standard basis vector of $a$ in $\reals^{V_G}$.

Consider
\begin{equation}\label{eq_defines_W_as_M_one_f_k}
W=\{ M_1 \mec f_k \ | \ \mbox{$M_1$ is supported on $G$} \} 
\subset \reals^{V_G};
\end{equation} 
let us show that
\begin{equation}\label{eq_W_contains_basis_vecs_of_NonUrsh_and_shallow}
v\in{\rm NonUrschel}(\mec f_k)\cup{\rm Shallow}(\mec f_k)
\quad\implies\quad
\mec e_v\in W.
\end{equation} 
To see this, first note 
for all $v\in{\rm NonUrschel}(\mec f_k)$, $\mec e_v\in W$, since we may take $M_1$
to be a diagonal matrix with a single nonzero diagonal entry at
$(v,v)$.
Next, for $v\in{\rm Shallow}(\mec f_k)$, there is a $v_1\in{\rm NonUrschel}(\mec f_k)$
that is adjacent to $v$.
To see that $\mec e_{v}\in W$, 
we can take $M_1$ to be everywhere zero except for
$M_1(v,v_1)=M_1(v_1,v)=1$ so that $M_1 \mec f_k$ is nonzero at and only
at $v$.
This establishes \eqref{eq_W_contains_basis_vecs_of_NonUrsh_and_shallow}.

Since $W$ is clearly a subspace of $\reals^{V_G}$,
\eqref{eq_W_contains_basis_vecs_of_NonUrsh_and_shallow} implies
\begin{equation}\label{eq_no_deep_Urschel_points_get_everything}
\realBVG{{\rm NonUrschel}(\mec f_k)}+
\realBVG{{\rm Shallow}(\mec f_k)}
\subset W 
\end{equation} 
(recall the notation \eqref{eq_reals_B_subset_A}).
[We also see that the above is an equality, since $M_1$ is supported on
the diagonal and the edges of $G$;
hence for any deep Urschel vertex, $v$, we have that $v$ is not
adjacent to any non-Urschel vertex, nor is it one itself, and hence
$(M_1f)(v)=0$ for all $M_1$ supported on $G$.
{\myred Hence \eqref{eq_no_deep_Urschel_points_get_everything} 
holds with equality.}]

{\myred
Since $\mec f_k$ vanishes on all Urschel vertices of $\mec f_k$,
we have
\begin{equation}\label{eq_f_k_perp_equals_U_one_and_reals_to_the_Urschel}
\mec f_k^\perp = U_1 + \realBVG{{\rm Urschel}(\mec f_k)} ,
\end{equation} 
where 
$$
U_1 = \mec f_k^\perp \cap \realBVG{{\rm NonUrschel}(\mec f_k)}
\subset \realBVG{{\rm NonUrschel}(\mec f_k)}.
$$
Since $\mec f_k$ spans a one-dimension subspace,
$$
\dim(U_1) = \bigr|{\rm NonUrschel}(\mec f_k) \bigl| - 1.
$$

Next set
$$
W_1 = \mec f_k^\perp \cap W.
$$
Since \eqref{eq_W_contains_basis_vecs_of_NonUrsh_and_shallow} holds
with equality,
$$
W_1 = U_1 + \realBVG{{\rm Shallow}(\mec f_k)},
$$
so
\begin{equation}\label{eq_W_one_formula_nonUrsch_minus_one_plus_s}
\dim(W_1) = \dim(U_1) + s =
\bigr|{\rm NonUrschel}(\mec f_k) \bigl| - 1+s.
\end{equation} 
}

Let us prove that 
\begin{equation}\label{eq_pseudoinverse_W_not_in_U}
(M_0-\lambda_k I)^+ W \not\subset 
\realBVG{{\rm NonUrschel}(\mec f_k)}
\end{equation} 
To prove this, first note that
$(M_0-\lambda_k I)^+$ is 
{\myred a bijection from $\mec f_k^\perp$ to itself.
Since $W_1\subset \mec f_k^\perp$,
$$
\dim\bigl( (M_0-\lambda_k I)^+ W_1 ) = \dim(W_1) = 
|{\rm NonUrschel}(\mec f_k)| -1 +s.
$$
Also
$$
{\rm Image}\bigl( (M_0-\lambda_k I)^+ \bigr) \cap 
\realBVG{{\rm NonUrschel}(\mec f_k)} = U_1
$$
is of dimension $|{\rm NonUrschel}(\mec f_k)| -1$.
Since $s\ge 1$ we therefore have
$$
(M_0-\lambda_k I)^+ W_1 \not\subset U_1,
$$
So 
$$
(M_0-\lambda_k I)^+ W \subset 
\realBVG{{\rm NonUrschel}(\mec f_k)}
$$
implies
$$
(M_0-\lambda_k I)^+ W_1 \subset U_1,
$$
which is impossible.  Hence 
\eqref{eq_pseudoinverse_W_not_in_U} holds.
}

In view of \eqref{eq_pseudoinverse_W_not_in_U}
and the definition \eqref{eq_defines_W_as_M_one_f_k} of $W$,
there exists $M_1$ supported on $G$ satisfying
$$
(M_0-\lambda_k I)^+ M_1 \mec f_k \notin \realBVG{{\rm NonUrschel}(\mec f_k)}.
$$
Then $(M_0-\lambda_k I)^+ M_1 \mec f_k$ is nonzero at at least one Urschel
vertex.
This proves the first part of
Lemma~\ref{le_perturbation_at_simple_eigenvalue_main}.


For the second part of the theorem, if there are no deep Urschel points then
{\myred
\eqref{eq_W_one_formula_nonUrsch_minus_one_plus_s} implies
$$
\dim(W_1)=|{\rm NonUrschel}(\mec f_k)|-1+s = |V_G|-1.
$$
Since this is the dimension of $\mec f_k^\perp$, we have
$W_1=\mec f_k^\perp$.
Hence
$$
(M_0-\lambda_k I)^+ W_1 = W_1 = \mec f_k^\perp,
$$
so $\mec f_k^1$ can be taken to have any value on $\mec f_k^\perp$,
and in particular any values on the Urschel vertices of $\mec f_k^\perp$.
}
{\myred This proves the second assertion.

Let $\pi\from \reals^{V_G}\to\reals^{{\rm Urschel}(\mec f_k)\subset V_G}$
be the orthogonal
projection (that zeros all ${\rm NonUrschel}(\mec f_k)$ components); consider
$$
P = \pi \bigl( (M_0-\lambda_k I)^+ W_1\bigr)
\subset \reals^{{\rm Urschel}(\mec f_k)\subset V_G} \subset \reals^{V_G}.
$$
Since the 
$$
{\rm Image}\bigl( (M_0-\lambda_k I)^+ \bigr)
= U_1 + \reals^{{\rm Urschel}(\mec f_k)\subset V_G},
$$
we have
$$
(M_0-\lambda_k I)^+ W_1 \subset U_1 + P.
$$
Hence 
\begin{equation}\label{eq_dim_P_lower_bound}
\dim(P) \ge \dim\bigl( (M_0-\lambda_k I)^+ W_1 \bigr) - \dim(U_1)
= \dim(W_1)-\dim(U_1)
\end{equation} 
since $(M_0-\lambda_k I)^+$ is a bijection on $\mec f_k^\perp$
and $W_1\subset \mec f_k^\perp$.
Putting this into 
\eqref{eq_W_one_formula_nonUrsch_minus_one_plus_s} gives
$$
\dim(P) \ge \bigl( \dim(U_1)+s \bigr) - \dim(U_1)=s.
$$
The elements of $P$ can be described by having
$\dim(P)$ free coordinates in ${\rm Urschel}(\mec f_k)$,
and the rest of the coordinates there fixed as functions of the
free coordinates (after reducing the rowspace of $P$ to echelon form).
Since $\dim(P)\ge s$, we have at least $s$ free variables which can 
be assigned arbitrary values and determine a vector $\mec p\in P$.
Then
$$
\pi \bigl( (M_0-\lambda_k I)^+ W_1\bigr) = P
$$
implies there is some vector in $(M_0-\lambda_k I)^+ W_1$ whose values
on ${\rm Urschel}(\mec f_k)$ are those of $\mec p$.
}
\end{proof}

We use the lemma above to prove
Theorems~\ref{th_max_urschel_number_simple_eig_all_shallow}--\ref{th_Urschel_most_general}.
All these theorems involve the case of a simple eigenvalue
(multiplicity 1) of a generalized Laplacian.

Recall Theorem~\ref{th_max_urschel_number_simple_eig_all_shallow}
deals with the case where there are no Urschel vertices.

\begin{proof}[Proof of 
Theorem~\ref{th_max_urschel_number_simple_eig_all_shallow}]
Let $\mec f_k^1\in\reals^{V_G}$ be any function such that $\mec f_k^1(v)=0$ for
all $v\in{\rm NonUrschel}(\mec f_k)$.
Then 
Lemma~\ref{le_perturbation_at_simple_eigenvalue_main}
shows that for some $M_1$ we have
$M(\epsilon)=M_0+\epsilon M_1$ has its $k$-th eigenvector equal to
$\mec f_k(\epsilon)=\mec f_k+\epsilon \mec f_k^1 + O(\epsilon^2)$ for 
$|\epsilon|$
sufficiently small.  Taking $\epsilon>0$ and sufficiently small
we have that $\mec f_k(\epsilon)$ has the sign of $\mec f_k$ at its
non-Urschel vertices, and the sign of $\mec f_k^1$ at its Urschel vertices.
Hence for $\epsilon>0$ sufficiently small, $\mec f_k(\epsilon)$ is a signing
of $\mec f_k$ with arbitrarily prescribed sign pattern.
But for $\epsilon>0$ sufficiently small, $M(\epsilon)=M_0+\epsilon M_1$
is a generalized Laplacian whose $k$-th eigenvalue is $\lambda_k(\epsilon)$;
hence
Theorem~\ref{th_friedmans_nodal_region_theorem} implies that
${\rm SND}(\mec f_k(\epsilon))\le k$.
Hence $\MUN(\mec f_k)\le k$.
\end{proof}

Recall that Theorem~\ref{th_Urschel_reproved_with_two_signings}
deals with the case where there are Urschel vertices.

Note that if $v$ is an Urschel vertex of $\mec f_k$ at which 
$\mec f_k^1$ is nonzero,
then for small $\epsilon>0$,
$f_k(\epsilon)(v)$ and $f_k(-\epsilon)(v)$ have the opposite sign.

\begin{proof}[Proof of Theorem~\ref{th_Urschel_reproved_with_two_signings}]
{\myred
By Lemma~\ref{le_perturbation_at_simple_eigenvalue_main} we may choose
$M_1$ that is supported on $G$ such that
$\mec f_k^1$ given by
\eqref{eq_f_sub_k_sup_one_with_pseudoinverse} is nonzero at one or more
Urschel vertices of $\mec f_k$.
Fix such an $M_1$, let $\mec f_k^1$ be as 
in \eqref{eq_f_sub_k_sup_one_with_pseudoinverse}.
It follows that for $\epsilon>0$ sufficiently small we have:
\begin{enumerate}
\item
$M(\epsilon)=M_0+\epsilon_1 M_1$ is a generalized Laplacian;
\item
$\mec f_k(\epsilon)=\mec f_k+\epsilon\mec f_k^1+O(\epsilon^2)$ is the
$k$-th eigenvalue of $M(\epsilon)$; 
\item
$\mec f_k(\epsilon)$
has the same sign as $\mec f_k$ wherever $\mec f_k$ is nonzero; and
\item
$\mec f_k(\epsilon)$
has the same sign as $\mec f_k^1$ wherever $\mec f_k$ is zero and
$\mec f_k^1$ is nonzero.
\end{enumerate}
Fix such an $\epsilon>0$, call it $\epsilon_1$, and let
$$
M_+^1=M(\epsilon_1), \quad \mec g_{k,1,+}=\mec f_k(\epsilon_1).
$$
Then $\mec g_{k,1,+}$ is a partial signing of $\mec f_k$.

If $g_{k,1,+}$ has no Urschel vertices, then $g_{k,1,+}$ is a signing
of $\mec f_k$, and by Theorem~\ref{th_friedmans_nodal_region_theorem},
${\rm SND}(g_{k,+}^1)\le k$.
Hence
\begin{equation}\label{eq_UN_one_f_k_at_most_k}
{\rm UN}_1(\mec f_k)\le k.
\end{equation} 
If, by contrast, $g_{k,1,+}$ has Urschel vertices, then we will
still prove \eqref{eq_UN_one_f_k_at_most_k} as follows: repeat the argument 
in the
last paragraph to find
$$
M_+^2=M(\epsilon_2), \quad \mec g_{k,2,+}=\mec f_k(\epsilon_2)
$$
so that (1) $\mec g_{k,2,+}$ is a partial signing of $\mec g_{k,1,+}$ with
at least one less Urschel vertex, 
and (2) $\mec g_{k,2,+}$ is the $k$-th eigenfunction of the generalized
Laplacian $M_+^2$.
We continue this procedure to find $M_+^i,\mec g_{k,i,+}$
with $i=3,4,\ldots$ until we reach an $i=\ell$ at which point
$\mec g_{k,\ell,+}$ has no Urschel vertices.
Then ${\rm SND}(\mec g_{k,\ell,+})\le k$, and~\eqref{eq_UN_one_f_k_at_most_k}
holds because for all $1\le i\le \ell-1$,
$\mec g_{k,i+1,+}$ is a partial signing of $\mec g_{k,i,+}$.

Hence \eqref{eq_UN_one_f_k_at_most_k} holds.
Now we repeat the same argument in the first paragraph, but we choose
$\epsilon_1<0$ and sufficiently small,
so that conditions~(1)--(2) in the first paragraph hold except that
$\mec f_k(\epsilon_1)$ has the {\em opposite} sign as $\mec f_k^1$
wherever it is nonzero and $\mec f_k$ is zero.
We set 
$$
M_-^1=M(\epsilon_1), \quad \mec g_{k,1,-}=\mec f_k(\epsilon_1),
$$
and note that $g_{k,1,-}$ is a different partial signing of $\mec f_k$ 
than is $g_{k,1,+}$, since $g_{k,1,-}$ and $g_{k,1,+}$ have opposite
sign wherever $\mec f_k^1$ is nonzero and $\mec f_k$ is zero.
Continuing to find $M_-^i,\mec g_{k,i,-}$ for $i=2,3,\ldots$,
we eventually reach $\mec g_{k,\ell',-}$ without any Urschel vertices.
Since $\mec g_{k,\ell',-}$ is a different signing from $g_{k,\ell,+}$,
we have
$$
{\rm UN}_2(\mec f_k)\le k.
$$
}
\end{proof}

{\myred 
Recall that Theorem~\ref{th_Urschel_most_general}
generalizes the main result of the two former theorems.

\begin{proof}[Proof of Theorem~\ref{th_Urschel_most_general}]
Same argument as in the proof of 
Theorem~\ref{th_Urschel_reproved_with_two_signings}, noticing that
we may choose $\mec f_k^1$ to have any signing on at least $s$
Urschel vertices of $\mec f_k$.
Choosing all possible $2^s$ sign patterns on these vertices gives
us at leaset $2^s$ different partial signings of $\mec f_k$,
each of which eventually gives a signing of $\mec f_k$ with
at most $k$ strong nodal domains.  Hence
$$
{\rm UN}_{2^s}(\mec f_k) \le k.
$$
\end{proof}
}
\section{Perturbation of Generalized Laplacians at a Multiple Eigenvalue}
\label{se_gen_Laplacian_perturb_multiple}

In this section we describe perturbations of generalized Laplacians.
We focus on a theorem regarding a multiple eigenvalue of a generalized
Laplacian.

\begin{lemma}\label{le_perturbation_at_multiple_eigenvalue_main}
Let $G$ be a finite, simple graph, and $M_0$ any generalized Laplacian
on $G$ with eigenvalues 
\eqref{eq_eigenvalues_increasing_order}.
Say that $\lambda_k$ has multiplicity $m$
(i.e., \eqref{eq_kth_eigenvalue_has_multiplicity_m_intro} holds).
Then there exists a perturbation of $M_0$,
$$
M(\epsilon)=M_0 + \epsilon M_1 
$$
and $\lambda_k^{1,1}<\lambda_k^{2,1}<\cdots<\lambda_k^{m,1}$
such that for $|\epsilon|$ nonzero and sufficiently small,
$M(\epsilon)$ has $m$ distinct eigenvalues
$$
\lambda_k^j(\epsilon)=\lambda_k + \epsilon\lambda_k^{j,1} + O(\epsilon^2),
\quad j\in[m]
$$
which are the $k$-th through $(k+m-1)$-th eigenvalues
of $M(\epsilon)$, and each $\lambda_k^j(\epsilon)$
has an eigenfunction
$$
\mec f_k^j(\epsilon)=\mec f_k^{j,0}+ \epsilon \mec f_k^{j,1}+O(\epsilon^2).
$$
(We write $\mec f_k^j$ and $\lambda_k^j$ with superscripts because the 
order of the $\lambda_k^j(\epsilon)$ depends on whether $\epsilon>0$ or 
$\epsilon<0$.)
Moreover, $\mec f_k^{1,0},\ldots,\mec f_k^{m,0}$ are mutually orthogonal.
\end{lemma}
\begin{proof}[Proof of 
Lemma~\ref{le_perturbation_at_multiple_eigenvalue_main}]
By definition, any diagonal matrix $M_1$ yields a 
perturbation of $M$, $M(\epsilon)=M+\epsilon M_1$.
Now apply Lemma~\ref{le_perturb_subspace_E_at_a_diagonal_matrix}
and Theorem~\ref{th_perturbation_theory_at_a_multiple_eigenvalue}.
\end{proof}

\begin{proof}[Proof of 
Theorem~\ref{th_perturbation_at_multiple_eigenvalue_main}]
Apply
Lemma~\ref{le_perturbation_at_multiple_eigenvalue_main};
it follows that for some
$\epsilon_0>0$ sufficiently small we have
for all $0<\epsilon<\epsilon_0$:
\begin{enumerate}
\item
$$
\lambda_k^1(\epsilon) < \lambda_k^2(\epsilon) < \cdots <
\lambda_k^m(\epsilon),
$$
and
\item
for each $j$, 
$\mec f_k^j(\epsilon)$ and $\mec f_k^{j,0}$
have the same sign at all non-Urschel points 
of $\mec f_k^{j,0}$
\end{enumerate}
(such an $\epsilon_0$ exists for item~(1) since
$\lambda_k^j(\epsilon)=\lambda_k + \epsilon\lambda_k^{j,1}+O(\epsilon^2)$,
and for item~(2) since $\mec f_k^j(\epsilon)=\mec f_k^{j,0}+O(\epsilon)$).

Take any $\epsilon\in(0,\epsilon_0)$: 
since
$$
\lambda_k^1(\epsilon) < \lambda_k^2(\epsilon) < \cdots <
\lambda_k^m(\epsilon) ,
$$
each of these eigenvalues are simple, and then 
Theorem~\ref{th_Urschel_reproved_with_two_signings} implies that
\begin{equation}\label{eq_UN_f_k_j_epsilon}
{\rm UN}(\mec f_k^j(\epsilon))\le k+j-1.
\end{equation} 
Since $\mec f_k^{j,0}$ and $\mec f_k^j(\epsilon)$ agree in sign on all
Urschel vertices of $\mec f_k^{j,0}$,
$\mec f_k^j(\epsilon)$ is a partial signing of $\mec f_k^{j,0}$, and therefore
$$
{\rm UN}(\mec f_k^{j,0}) \le {\rm UN}(\mec f_k^j(\epsilon)).
$$
Combining this with \eqref{eq_UN_f_k_j_epsilon} gives
\begin{equation}\label{eq_UN_f_k_sup_j_zero_bound}
{\rm UN}(\mec f_k^{j,0}) \le k+j-1.
\end{equation} 

Similarly there exists $\epsilon_0$ sufficiently small such that if
$-\epsilon_0<\epsilon<0$ then items~(1) and~(2) above hold with~(1)
replaced by
$$
\lambda_k^m(\epsilon) < \lambda_k^{m-1}(\epsilon) < \cdots <
\lambda_k^1(\epsilon)
$$
(since $\epsilon$ is small and {\em negative}).
Then for $j=0,\ldots,m-1$, $\lambda_k^{m-1-j}(\epsilon)$ is the
$(k+j)$-th largest eigenvalue, and hence we similarly get
$$
{\rm UN}(\mec f_k^{m-1-j,0}) \le k+j-1,
$$
and therefore for $j\in[m]$ we have
\begin{equation}\label{eq_UN_f_k_sup_j_zero_bound_in_other_direction}
{\rm UN}(\mec f_k^{j,0}) \le k+m-j.
\end{equation} 

So now take
$$
\mec f_k=\mec f_k^{1,0},\ \mec f_{k+1}=\mec f_k^{2,0},\ \ldots, \mec f_{k+m-1}=\mec f_k^{m,0}.
$$
It follows that $\mec f_k,\ldots,\mec f_{k+m-1}$ is an orthonormal basis for
the eigenspace $E(M,\lambda_k)$,
and
\eqref{eq_UN_f_k_sup_j_zero_bound} and
\eqref{eq_UN_f_k_sup_j_zero_bound_in_other_direction} imply
\eqref{eq_UN_multiple_eigenvalue_two_sided_bound}.
\end{proof}

\section{The Ladder, Revisited}
\label{se_ladder_again}

Consider again the $1=\lambda_3=\lambda_4$ eigenspace of the
classical Laplacian of the ladder graph:

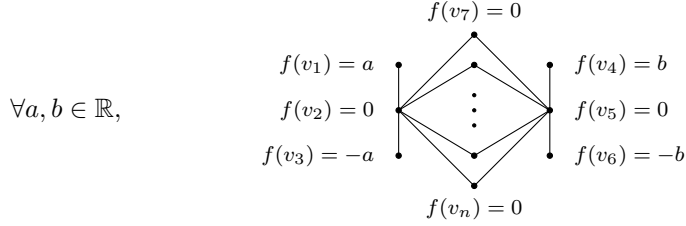
\begin{tikzpicture}[scale=0.20]
\node at (-22,0) {$\forall a,b\in\reals$,};
\tikzmath{ 
  int \j ;
}
\foreach \i in {1,...,3} {
  \filldraw (0,6 - \i * 3) circle (5pt);
  \filldraw (10,6 - \i * 3) circle (5pt);
  \tikzmath{
    \j = \i + 3;
  }
  \filldraw (5, 2 - \i * 1 ) circle (3pt);
}
\node[anchor=east]  at (-1,3) {\Small $f(v_1)=a$};
\node[anchor=east] at (-1,0) {\Small $f(v_2)=0$};
\node[anchor=east] at (-1,-3) {\Small $f(v_3)=-a$};
\node[anchor=west] at (11,3) {\Small $f(v_4)=b$};
\node[anchor=west] at (11,0) {\Small $f(v_5)=0$};
\node[anchor=west] at (11,-3) {\Small $f(v_6)=-b$};
\draw (0,3) to (0,-3);
\draw (10,3) to (10,-3);
\filldraw (5,5) circle (5pt);
\node at (5,6.5) {\Small $f(v_7)=0$};
\filldraw (5,-5) circle (5pt);
\node at (5,-6.5) {\Small $f(v_n)=0$};
\draw (0,0) to (5,5);
\draw (10,0) to (5,5);
\draw (0,0) to (5,-5);
\draw (10,0) to (5,-5);
\filldraw (5,3) circle (5pt);
\filldraw (5,-3) circle (5pt);
\draw (0,0) to (5,3);
\draw (10,0) to (5,3);
\draw (0,0) to (5,-3);
\draw (10,0) to (5,-3);
\end{tikzpicture}

It is interesting to compare the perturbation method with
Urschel's algorithm on this example.

Consider $f_1\in E(\Delta_G,\lambda_3)$ to be the
case $a=1$ and $b=0$, and similarly for $f_2$
and $a=0$ and $b=1$.
Then $f_1$ and $f_2$ are orthogonal; also,
we easily see that for any symmetric $M_1$ that is supported on $G$ we
have $(M_1f_1,f_2)=0$.  It follows that $PM_1 f_1$ is proportional to $f_1$,
and $PM_1 f_2$ to $f_2$.
Hence for any such $M_1$, if $PM_1$ has distinct eigenvalues, then
the eigenvectors are necessarily $f_1,f_2$.
Hence perturbation theory does not yield a full measure subset of
$E(\Delta_G,\lambda_3)$ as possible lowest eigenvalues.
However, it is interesting
that for $i=1,2$ we have
$$
{\rm WND}(f_i)={\rm UN}(f_i)={\rm SND}(f_i)=2 < 3,
$$
since the corresponding eigenvalues are $\lambda_3=\lambda_4=1$.

Note that the ladder also points out a shortcoming of the perturbation
method: namely, perturbation theory --- as done in this article ---
will only find the eigenvectors
with $a=0$ or $b=0$; therefore the perturbation theory never finds a 
full measure subset of the eigenvectors of $\lambda_3$.

\newcommand{\ignoreThisInFinalVersion}[1]{}
\ignoreThisInFinalVersion{
By contrast, Urschel's method looks at the connected components of
(1) subgraph of $G$ induced 
on the non-Urschel vertices, $\cN=\{v_1,v_3,v_4,v_6\}$, of $E(\Delta,1)$,
and 
(2) the same for the Urschel vertices, $\cU=\{v_2,v_5,v_7,\ldots,v_n\}$.
See Figure~\ref{fi_Urschel_nonUrschel_vertices_ladder}.
\begin{figure}[ht]
$$
\begin{array}{cc}
\begin{tikzpicture}[scale=0.20]
\tikzmath{ 
  int \j ;
}
\foreach \i in {1,...,3} {
  \filldraw (0,6 - \i * 3) circle (5pt);
  \filldraw (10,6 - \i * 3) circle (5pt);
  \tikzmath{
    \j = \i + 3;
  }
  \filldraw (5, 2 - \i * 1 ) circle (3pt);
}
\node[anchor=east]  at (-1,3) {\Small $f(v_1)=a$};
\node[anchor=east] at (-1,0) {\Small $f(v_2)=0$};
\node[anchor=east] at (-1,-3) {\Small $f(v_3)=-a$};
\node[anchor=west] at (11,3) {\Small $f(v_4)=b$};
\node[anchor=west] at (11,0) {\Small $f(v_5)=0$};
\node[anchor=west] at (11,-3) {\Small $f(v_6)=-b$};
\draw (0,3) to (0,-3);
\draw (10,3) to (10,-3);
\filldraw (5,5) circle (5pt);
\node at (5,6.5) {\Small $f(v_7)=0$};
\filldraw (5,-5) circle (5pt);
\node at (5,-6.5) {\Small $f(v_n)=0$};
\draw (0,0) to (5,5);
\draw (10,0) to (5,5);
\draw (0,0) to (5,-5);
\draw (10,0) to (5,-5);
\filldraw (5,3) circle (5pt);
\filldraw (5,-3) circle (5pt);
\draw (0,0) to (5,3);
\draw (10,0) to (5,3);
\draw (0,0) to (5,-3);
\draw (10,0) to (5,-3);
\node at (5,-10) {\Small The subspace $E(\Delta_G,1)$};
\end{tikzpicture}
\quad&\quad
\begin{tikzpicture}[scale=0.20]
\tikzmath{ 
  int \j ;
}
\foreach \i in {1,3} {
  \filldraw (0,6 - \i * 3) circle (5pt);
  \node at (-1.5,6 - \i * 3) {\Small $v_\i$};
  \filldraw (10,6 - \i * 3) circle (5pt);
  \tikzmath{
    \j = \i + 3;
  }
  \node at (11.5,6 - \i * 3) {\Small $v_\j$};
}
%
\node at (5,-10) {\Small Induced subgraph on, $\cN$, the non-Urschel vertices};
\end{tikzpicture}
\\
\begin{tikzpicture}[scale=0.20]
\tikzmath{ 
  int \j ;
}
\foreach \i in {2} {
  \filldraw (0,6 - \i * 3) circle (5pt);
  \node at (-1.5,6 - \i * 3) {\Small $v_\i$};
  \filldraw (10,6 - \i * 3) circle (5pt);
  \tikzmath{
    \j = \i + 3;
  }
  \node at (11.5,6 - \i * 3) {\Small $v_\j$};
  \filldraw (5, 2 - \i * 1 ) circle (3pt);
}
\filldraw (5, 2 - 1 * 1 ) circle (3pt);
\filldraw (5, 2 - 3 * 1 ) circle (3pt);
\filldraw (5,5) circle (5pt);
\node at (5,6.5) {\Small $v_7$};
\filldraw (5,-5) circle (5pt);
\node at (5,-6.5) {\Small $v_n$};
\draw (0,0) to (5,5);
\draw (10,0) to (5,5);
\draw (0,0) to (5,-5);
\draw (10,0) to (5,-5);
\filldraw (5,3) circle (5pt);
\filldraw (5,-3) circle (5pt);
\draw (0,0) to (5,3);
\draw (10,0) to (5,3);
\draw (0,0) to (5,-3);
\draw (10,0) to (5,-3);
\node at (5,-10) {\Small Induced subgraph on, $\cU$, the Urschel vertices};
\end{tikzpicture}
\quad&\quad
\begin{tikzpicture}[scale=0.20]
\tikzmath{ 
  int \j ;
}
\foreach \i in {1,2,3} {
  \filldraw (0,6 - \i * 3) circle (5pt);
  \filldraw (10,6 - \i * 3) circle (5pt);
  \tikzmath{
    \j = \i + 3;
  }
}
\draw[dashed] (0,3) to (0,-3);
\draw[dashed] (10,3) to (10,-3);
\node[anchor=east]  at (-1,3) {\Small $f(v_1)=\alpha_1$};
\node[anchor=east]  at (-1,0) {\Small Shallow Urschel $v_2$};
\node[anchor=east] at (-1,-3) {\Small $f(v_3)=\alpha_2$};
\node[anchor=west] at (11,3) {\Small $f(v_4)=\alpha_3$};
\node[anchor=west]  at (11,0) {\Small Shallow Urschel $v_5$};
\node[anchor=west] at (11,-3) {\Small $f(v_6)=\alpha_4$};
\node at (5,-6) {\Small Urschel's construction: the two shallow};
\node at (5,-8) {\Small vertices yield two linear equations};
\node at (5,-10) {\Small $\alpha_1+\alpha_2=0$ and $\alpha_3+\alpha_4=0$};
\end{tikzpicture}
\label{fi_Urschel_nonUrschel_vertices_ladder}
\end{array}
$$
\caption{Urschel's method: look at the connected components of $G$'s
subgraph induced on (1) $\cN$, the non-Urschel vertices, and (2)
$\cU$, the Urschel vertices.  Hence (1) has 4 connected components
(each an isolated vertex), and (2) has a single connected component.
Here $M$ restricted to $\cN\times\cN$ is the identity matrix, so
$\lambda=1$ is an eigenvalue of multiplicity $4$ there.
Each shallow Urschel vertex, here $v_2$ and $v_5$, determines a linear
equation on $\alpha_1,\ldots,\alpha_4$, which Urschel writes as
$\alpha_1=-\alpha_2$, $\alpha_3=-\alpha_4$ with $\alpha_1,\alpha_3$ the
fixed variables in terms of $\alpha_2,\alpha_4$ the free variables.
Urschel's algorithm
chooses $\alpha_2>0$, then chooses one sign for the single connected
component of Urschel variables, then chooses $\alpha_4$ to be of this
same sign, thereby reducing connecting two of the 4 connected components in
$\cN$ through this Usrchel component.
Because one only has sign requirements on $\alpha_2$ and $\alpha_4$, we
get a full dimension subset of $E(\Delta_G,1)$.}
\end{figure}
and $M$ restricted to $\cN\times\cN$ is
the identity matrix.  His method begins with
this $4$-dimensional set of eigenvectors with eigenvalue $1$,
namely all of $\reals^\cN$; it follows that any $f\in\reals^{\cN}$,
if extended by 0 to the Urschel vertices, would have four strong
nodal regions.
Urschel considers the linear
conditions on $\reals^\cN$ required to create a zero sum at each
Urschel vertex, of which there are two conditions (since there are
two shallow Urschel vertices).  This gives the two-dimensional
space of functions in $f\in\reals^\cN$ such that 
$f(v_1)=-f(v_3)$ and $f(v_4)=-f(v_6)$, so $f(v_3),f(v_6)$ are viewed
as free variables to this system of two equations.
He then argues that by choosing the appropriate signs for the
pivot variables and choosing a sign (or ``charge'') at the
Urschel vertices, the four original nodal regions is reduced to
three.
Since the Urschel vertices are connected in $G$, Urschel's method
picks a single signing (or ``charge'') to all Urschel vertices
at once.
Hence Urschel assigns a sign to the pivot variables in the
linear equations and signs on the Urschel vertices in a way that
drops the number of strong nodal domains from four to three.
Since all he needs is to specify the signs of the pivot variables,
Urschel's method shows that there is a full measure subset of
$E(\Delta_G,1)$ whose Urschel number is at most $3$.

Hence Urschel's method gives a much larger part of eigenvectors
in $E(\Delta_G,1)$ whose Urschel number is at most $3$.

By contrast, the perturbation method shows that 
both $f_1$ and $f_2$ above, which are orthogonal, have Urschel
number at most $3$.

Of course, in this example more is true: we easily see that
any element of $E(\Delta_G,1)$ has Urschel number at most $3$,
and $f_1,f_2$ above have Urschel number $2$.
}

\section{Examples with Arbitrary Number of Shallow and Deep Urschel Vertices}
\label{se_shallow_deep}

Let $s,k\in\naturals$ and $\ell\in\integers_{\ge 0}$.  We will build
a connected graph, $G$, and a generalized Laplacian, $M$, on $G$
such that $\lambda_2(M)$ is a simple eigenvalue
(multiplicity one), and such that the corresponding eigenvector
$\mec f_2$ has $s$ shallow Urschel vertices, $\ell$ deep Urschel vertices,
and $n=s+\ell+2k$ vertices.
We give a specific example first, namely the graph, $G$, depicted in 
Figure~\ref{fi_example_any_number_of_deep_and_shallow};
we then indicate a much more general class of examples.

\newcommand{\myFirstArbitraryNumberShallowDeep}[5]{
\begin{tikzpicture}[scale=0.30]
\foreach \i in {-1,1} {
  \filldraw (0, \i * #1) circle (5pt);
  \filldraw (-#2, \i * #1) circle (5pt);
  \node at (-#2*2, \i * #1) {$\cdots$};
  \filldraw (-#2*3, \i * #1) circle (5pt);
  \draw (0,\i * #1) -- (-#2*1.5,\i * #1);
  \draw (-#2*2.5,\i * #1) -- (-#2*3,\i * #1);
}
\node at (0, #1 + 1.2) {$v_1$};
\node at (-#2, #1 + 1.2) {$v_2$};
\node at (-#2*3, #1 + 1.2) {$v_k$};
\node at (0, -#1 - 1.2) {$v'_1$};
\node at (-#2, -#1 - 1.2) {$v'_2$};
\node at (-#2*3, -#1 - 1.2) {$v'_k$};
\filldraw (#2,0) circle (5pt);
\node at (#2-1.2,0) {$u_1$};
\draw (0,-#1) -- (#2,0) -- (0,#1);
\filldraw (#2-#3,0) circle (5pt);
\node at (#2-#3-1.2,0) {$u_2$};
\draw (0,-#1) -- (#2-#3,0) -- (0,#1);
\node at (#2-2*#3,0) {$\cdots$};
\filldraw (#2-3*#3,0) circle (5pt);
\node at (#2-3*#3-1.2,0) {$u_s$};
\draw (0,-#1) -- (#2-3*#3,0) -- (0,#1);
\filldraw (#2+#4,#5) circle (5pt);
\node at (#2+#4+1.2,#5) {$w_1$};
\draw (#2,0) -- (#2+#4,#5);
\node at (#2+#4,0.2) {$\vdots$};
\filldraw (#2+#4,-#5) circle (5pt);
\node at (#2+#4+1.2,-#5) {$w_\ell$};
\draw (#2,0) -- (#2+#4,-#5);
\end{tikzpicture}
}

\begin{figure}[ht]
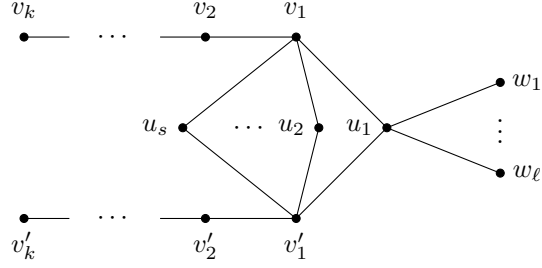

$$
\myFirstArbitraryNumberShallowDeep{4}{4}{3}{5}{2}
$$
\caption{The graph $G$: it consists of two paths, a top path with vertices
$v_1,\ldots,v_k$, and a bottom path $v'_1,\ldots,v'_k$; $v_1$ and $v'_1$ are
both connected to vertices $u_1,\ldots,u_s$, and $u_1$ is connected to
vertices $w_1,\ldots,w_\ell$.
Once we have described a generalized Laplacian, $M$ on $G$,
$\lambda_2(M)$ will be a simple eigenvalue (multiplicity 1), and for
the corresponding eigenvector $\mec f_2$:
the $v_i$ and $v'_i$ will be the non-Urschel;
the $u_i$ will be the shallow Urschel vertices;
and the $w_i$ will be the deep Urschel vertices.}
\label{fi_example_any_number_of_deep_and_shallow}
\end{figure}

So consider this graph, and let
$M=M(\mu)=\Delta_G+\mu I_U-(s-1) I_1$, where 
$\Delta_G$ is the graph Laplacian, $\mu\in\reals$,
$I_U$ is the identity matrix restricted to $u_1,\ldots,u_s$ and
is zero elsewhere, and $I_1$ is the identity matrix restricted to
$v_1,v_1'$.  Hence we may write $M$ in block form as
\begin{equation}\label{eq_M_mu_in_block_form}
M(\mu) =
\begin{bmatrix} A + \mu I & B \\ C & D \end{bmatrix} ;
\end{equation} 
since $M(\mu)$ is symmetric, we have $A,D$ are symmetric and
$B=C^{\rm T}$.
We now consider $M(\mu)$ with positive $\mu\in\reals$ large.

\subsection{The Eigenvalues of $M(\mu)$ with $\mu$ Large}

\begin{proposition}
Let $M=M(\mu)$ by any symmetric matrix of the
form \eqref{eq_M_mu_in_block_form}, with $A,B,C,D$ fixed
(we stick to the convention that $A$ is of size $s\times s$ and
$D$ is $(\ell+2k)\times(\ell+2k)$).
As $\mu\to\infty$, $M(\mu)$ has the following eigenvalues:
\begin{enumerate}
\item 
$s$ eigenvalues equal to $\mu+\alpha_i+O(1/\mu)$, where
$\alpha_1,\ldots,\alpha_s$ are the eigenvalues of $A$; and
\item 
$\ell+2k$ eigenvalues equal to $\beta_i+O(1/\mu)$, where
$\beta_1,\ldots,\beta_{\ell+2k}$ are the eigenvalues of $D$.
\end{enumerate}
\end{proposition}
This is a standard type of result; physically 
this corresponds to having infinite
(positive) potential at the $U$ vertices.

To prove this we will use the Schur complement formulas
\begin{equation}\label{eq_first_Schur_complement_formula}
\det\begin{bmatrix} A & B \\ C & D \end{bmatrix} =
\det[A] \det[D-CA^{-1}B]
\end{equation} 
if $A$ is invertible, and
\begin{equation}\label{eq_second_Schur_complement_formula}
\det\begin{bmatrix} A & B \\ C & D \end{bmatrix} =
\det[D] \det[A-BD^{-1}C]
\end{equation} 
if $D$ is invertible.
Plus we use the fact that 
if $A,A'$ are two symmetric matrices, then the eigenvalues
of $A$ and of $A+A'$ differ as sets from one another by $\|A'\|_{L^2}$
(the $L^2$-operator norm of $A'$, i.e., its spectral radius).
\begin{proof}
The eigenvalues of $M(\mu)$, which are the roots, $\lambda$,
of the equation:
\begin{equation}\label{eq_characteristic_equations_M_mu}
\det\bigl(\lambda I- M(\mu) \bigr)
=
\det
\begin{bmatrix} (\lambda-\mu)I - A  & -B \\ -C & \lambda I - D \end{bmatrix} .
\end{equation} 
Let us search for $s$ roots
of the form $\lambda=\mu+\nu$, where $|\nu|$ is bounded
by the spectral radius of $A$ plus $1$ (or some constant): we write the
right-hand-side of \eqref{eq_characteristic_equations_M_mu}
$$
\det
\begin{bmatrix} (\lambda-\mu)I - A  & -B \\ -C & \lambda I - D \end{bmatrix} 
=
\det
\begin{bmatrix}\nu I-A & -B \\ -C & \mu I + \nu I - D \end{bmatrix}
$$
which by \eqref{eq_second_Schur_complement_formula} equals
$$
=
\det[\mu I + \nu I - D]\det[\nu I-A - B(\mu I + \nu I - D)^{-1}C].
$$
For $\mu$ sufficiently large, $\det[\mu I + \nu I - D]$ is nonzero in our
range of $\nu$, and 
$$
(\mu I + \nu I - D)^{-1} = \mu^{-1}\bigl(I + (\nu I-D)/\mu \bigr)^{-1}
= \mu^{-1}\bigl(I + O(1/\mu)\bigr) .
$$
Hence the roots, $\nu$, of 
$$
\det[\nu I-A - B(\mu I + \nu I - D)^{-1}C] =
\det[\nu I - A - BC/\mu + O(1/\mu^2)]
$$
are the eigenvalues of matrix equal to $A-BC/\mu+O(1/\mu^2)$.
Hence --- since this matrix is symmetric ---
the eigenvalues of this matrix differs from that of $A$ by
$$
\bigl\| BC/\mu+O(1/\mu^2) \bigr\|_{L^2} \le O(1)/\mu
$$
as $\mu\to\infty$.
Setting $\lambda=\mu+\nu$ gives the first $s$ eigenvalues of the
form $\lambda=\mu+\alpha_i+O(1/\mu)$.

We similarly look for eigenvalues $\lambda$ with $\lambda$ bounded by
the spectral radius of $D$ plus 1, and similarly --- using 
\eqref{eq_first_Schur_complement_formula} --- find 
eigenvalues $\lambda=\beta_i+O(1/\mu)$.

Hence we have found $n=s+\ell+2k$ eigenvalues, which are all the eigenvalues
of $M(\mu)$ for $\mu$ sufficiently large.
\end{proof}

\subsection{The Eigenvalues of $D$}

Now consider the eigenvalues of $D$ in \eqref{eq_M_mu_in_block_form};
$D$ is the restriction of $M$ to the square submatrix indexed
$V\cup V'\cup W$ where
$$
V=\{v_1,\ldots,v_k\},
\ V'=\{v'_1,\ldots,v'_k\},
\ W=\{w_1,\ldots,w_\ell\}.
$$
Note that $G'$, the subgraph of $G$ induced on $V\cup V'\cup W$, has
$W$ as isolated vertices plus two discoonnected paths on the vertex
sets $V$ and $V'$.  It follows that $D=\Delta_{G'}+D'-(s-1)I_1$, where 
$D'$ is the diagonal matrix that counts how many $u_1,\ldots,u_s$ vertices
are indicent upon each vertex of $V\cup V'\cup W$;
therefore:
$D'(v_1,v_1)=D'(v'_1,v'_1)=s$, $D'(w_i,w_i)=1$ for all $i\in[\ell]$, and 
$D'(v_i,v_i)=D'(v'_i,v'_i)=0$ for all $2\le i\le k$.
It follows that $D'$ is the sum of the Laplacian on each path, plus
the identity matrix on $W$.
Hence the eigenvalues of $D$ are: $1$, with multiplicity $\ell$,
and two copies of the spectrum of Laplacian on a path of length $k$.
Hence $D$ has two eigenvalues that are zero, and the rest are bounded
away from zero, say $\ge c$ for a constant, $c>0$ (depending on $k$).

It follows that for $\mu$ sufficiently large, we have that
$M(\mu)$ has exactly two eigenvalues between $-c/3$ and $c/3$,
and the rest of the eigenvalues $\ge 2c/3$.

\subsection{Even and Odd Eigenfunctions}

Let $\sigma\from G\to G$ be the symmetry of $G$ that takes
$v_i$ to $v_i'$ and vice versa, and that fixes all other vertices.
By an {\em even function} (respectively, {\em odd function}) we mean a
function
$\mec f\from V_G\to\reals$ such that $\mec f\sigma=\mec f$
(respectively $\mec f\sigma=-\mec f$).
Then we easily see that
$M(\mu)$ takes even functions to even functions and odd to odd; hence
the spectrum of $M(\mu)$ decomposes to that on even functions and
that on odd functions.
Note that an odd function $\mec f\from V_G\to\reals$
has $f(v'_i)=f(v_i)$ for all $i$, and $\mec f=0$ on all other vertices.

For any value of $\mu$, since $G$ is connected we have that
$\lambda_1(M(\mu))$ is a simple eigenvalue and its eigenfunction
is (after scaling) positive on all vertices.  Hence this function
is an even function.

Now consider the odd function, $\mec f_2$, that is $1$ on $V$ and $-1$ on $V'$;
see Figure~\ref{fi_mec_f_two}.

\newcommand{\mySecondEigenfunctionForFirstArbitraryNumberShallowDeep}[5]{
\begin{tikzpicture}[scale=0.30]
\foreach \i in {-1,1} {
  \filldraw (0, \i * #1) circle (5pt);
  \filldraw (-#2, \i * #1) circle (5pt);
  \node at (-#2*2, \i * #1) {$\cdots$};
  \filldraw (-#2*3, \i * #1) circle (5pt);
  \draw (0,\i * #1) -- (-#2*1.5,\i * #1);
  \draw (-#2*2.5,\i * #1) -- (-#2*3,\i * #1);
}
\node at (0, #1 + 1.2) {$1$};
\node at (-#2, #1 + 1.2) {$1$};
\node at (-#2*3, #1 + 1.2) {$1$};
\node at (0, -#1 - 1.2) {$-1$};
\node at (-#2, -#1 - 1.2) {$-1$};
\node at (-#2*3, -#1 - 1.2) {$-1$};
\filldraw (#2,0) circle (5pt);
\node at (#2-1.2,0) {$0$};
\draw (0,-#1) -- (#2,0) -- (0,#1);
\filldraw (#2-#3,0) circle (5pt);
\node at (#2-#3-1.2,0) {$0$};
\draw (0,-#1) -- (#2-#3,0) -- (0,#1);
\node at (#2-2*#3,0) {$\cdots$};
\filldraw (#2-3*#3,0) circle (5pt);
\node at (#2-3*#3-1.2,0) {$0$};
\draw (0,-#1) -- (#2-3*#3,0) -- (0,#1);
\filldraw (#2+#4,#5) circle (5pt);
\node at (#2+#4+1.2,#5) {$0$};
\draw (#2,0) -- (#2+#4,#5);
\node at (#2+#4,0.2) {$\vdots$};
\filldraw (#2+#4,-#5) circle (5pt);
\node at (#2+#4+1.2,-#5) {$0$};
\draw (#2,0) -- (#2+#4,-#5);
\end{tikzpicture}
}

\begin{figure}[ht]
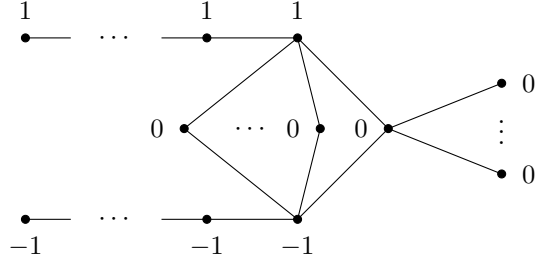

$$
\mySecondEigenfunctionForFirstArbitraryNumberShallowDeep{4}{4}{3}{5}{2}
$$
\caption{$\mec f_2$, which is the eigenfunction of the smallest odd
eigenvalue, and the second smallest eigenvalue overall.}
\label{fi_mec_f_two}
\end{figure}

We easily see that $M(\mu)\mec f_2=0$; hence
$\lambda_2(M(\mu))\le 0$, and this must be one of the two eigenvalues that
are $\le c/3$ in absolute value.
Hence for $\mu$ sufficiently large, $\lambda_2(M(\mu))=0>\lambda_1(M(\mu))$,
and so $\lambda_2(M(\mu))$ is a simple eigenvalue.

Now we see that each $u_i$ is a shallow Urschel vertex of $\mec f_2$,
and each $w_i$ is a deep Urschel vertex, and the $v_i$ and $v_i'$
are non-Urschel vertices.
It follows that $\mec f_2$ has $s$ shallow Urschel vertices, 
$\ell$ deep Urschel vertices, and $2k$ non-Urschel vertices.

\subsection{The Illustration of Theorem~\ref{th_Urschel_most_general}}

Now we can check that $\mec f_2$ has
$$
{\rm UN}_{2^s}(\mec f_2)\le 2,
$$
since we can take any sign
pattern for $\mec f_2$ on $u_1,\ldots,u_s$; as long as
the sign on each $w_i$ is that of $u_1$, we get exactly two nodal regions.
Notice also that if $\ell\ge 1$ then
$$
{\rm UN}_{2^s+1}(\mec f_2)=3,
$$
since given any sign pattern on $u_1,\ldots,u_s$, we get three or more
nodal regions if we do not have $w_1,\ldots,w_\ell$ having the same
sign as $u_1$ (and we get exactly three nodal regions by taking
$w_1,\ldots,w_{s-1}$ to have the same sign as $u_1$ and $w_s$ to have
the opposite sign).

Hence this is a case where in the inequality
$$
{\rm UN}_{2^s}(\mec f_k)\le k,
$$
we cannot improve $2^s$ to $2^s+1$.

\subsection{A More General Example}

The above example can be generalized as follows:
consider the graph, $G$, that is created from the following data:
we take two copies of the same connected graph, $H=(V,E)$ and $H'=(V',E')$
(i.e., $H$ and $H'$ are isomorphic).
We add vertices $U=\{u_1,\ldots,u_s\}$, such that each $u_i$ has
an edge to some vertices of $H$, and the same corresponding vertices in $H'$.
Finally we add vertices $W=\{w_1,\ldots,w_\ell\}$ here each 
$w_i$ is adjacent to any subset of the vertices of $U$.

Then $G$ has a symmetry $\sigma\from G\to G$ that exchanges $H$ and $H'$
and fixes the vertices in $U$ and $W$.
A generalized Laplacian similar to $M(\mu)$ above (with $D$ being the
sum of Laplacians of $H$ and $H'$, plus having a nonzero diagonal entry
on each $w_i$) therefore
has $\lambda_2$ as a simple eigenvalue with $\mec f_2$ (which is $1$ on $V_H$,
$-1$ on $V_{H'}$, and $0$ elsewhere) having
$s$ shallow Urschel vertices, $\ell$ deep Urschel vertices,
and $2|V_H|$ non-Urschel vertices.

We remark that if some $w_i\in W$ is adjacent to two Urschel vertices,
say $w_1$ is adjacent to $u_1$ and $u_2$,
then
$$
{\rm UN}_{2^s+1}(\mec f_2)=2
$$
because when $u_1$ and $u_2$ have opposite signs, $w_1$ can take on
either sign.
See Figure~\ref{fi_generalization_but_UN_two_to_the_s_not_tight}.

\newcommand{\myGeneralizationButUNtwoToTheSNotTight}[5]{
\begin{tikzpicture}[scale=0.30]
\foreach \i in {-1,1} {
  \filldraw (0, \i * #1) circle (5pt);
  \filldraw (4, \i * #1) circle (5pt);
  \filldraw (8, \i * #1) circle (5pt);
  \filldraw (12, \i * #1) circle (5pt);
  \filldraw (4, \i * #1 *1.8) circle (5pt);
  \draw (0, \i * #1) -- (12, \i * #1);
  \draw (0, \i * #1) -- (4, \i * #1 *1.8) -- (12, \i * #1);
  \draw (4, \i * #1) -- (4, \i * #1 *1.8);
}
\node at (-4,#1 *1.4) {$H$};
\node at (-4,-#1 *1.4) {$H'$};
\filldraw (-2,0) circle (5pt);
\draw (0,#1) -- (-2,0) -- (0,-#1);
\node at (-3.2,0) {$u_1$};
\filldraw (6,0) circle (5pt);
\node at (7.2,0) {$u_2$};
\draw (4,#1) -- (6,0) -- (4,-#1);
\draw (8,#1) -- (6,0) -- (8,-#1);
\filldraw (12,0) circle (5pt);
\node at (10.8,0) {$u_3$};
\filldraw (2,0) circle (5pt);
\node at (2,-1.2) {$w_1$};
\draw (12,#1) -- (12,-#1);
\draw (-2,0) -- (6,0);
\filldraw (15,2) circle (5pt);
\node at (16.2,2) {$w_2$};
\filldraw (15,-2) circle (5pt);
\node at (16.2,-2) {$w_3$};
\draw (12,0) -- (15,2);
\draw (12,0) -- (15,-2);
\end{tikzpicture}
}
\begin{figure}
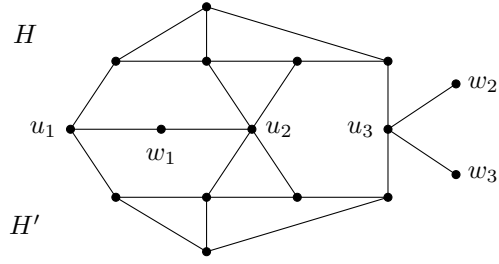

$$
\myGeneralizationButUNtwoToTheSNotTight{3}{0}{0}{0}{0}
$$
\caption{A Generalization of the first example.  In this case we
have attached $w_1$ to both $u_1$ and $u_2$.  So $w_1$ remains a
deep Urschel point for $\mec f_2$, but ${\rm UN}_{2^s+1}(\mec f_2)$
is also $2$ since when $u_1,u_2$ have opposite sign, $w_1$ can be
of either sign.  Hence, for this graph the bound
${\rm UN}_{2^s}(\mec f_2)\le 2$
is no longer tight.}
\label{fi_generalization_but_UN_two_to_the_s_not_tight}
\end{figure}

\section{Concluding Remarks}
\label{se_concluding}

{\myred

In this section, we make some concluding remarks.

\subsection{Eigenvalues Are Generically Simple}

In \cite{uhlenbeck}, Uhlenbeck shows that in analysis, a generic
Riemannian metric has its Laplacian eigenvalues simple.
The same is true in graph theory: indeed, if we fix a graph, $G$,
then a generalized Laplacian, $M$, on has simple roots iff
the resultant of its characteristic polynomial, $p=p(M)$, is nonzero.
Since a diagonal matrix can have distinct eigenvalues, it follows that
$p=p(M)$ restricted to generalized Laplacians is a nonzero polynomial
(when we set $m_{ij}=0$ when $i\ne j$ and $\{i,j\}\notin E_G$).
Since $p(M)$ is a nonzero polynomial, it is generically nonzero
(i.e., the set of $M$ on which $p(M)=0$ is a codimension $1$ subset
of all $M$ supported on $G$).

\subsection{Nodal Regions in Analysis}

In analysis, the number of nodal regions of an eigenfunction of a generalized
Laplacian on a closed surface does not increase under perturbations; see
\cite{saha_et_al_nodal_domain}, Theorem~B.
By contrast, for graphs, the number of nodal regions of an eigenfunction
of a generalized Laplacian can increase under perturbations.
For example, if in Figure~\ref{fi_example_any_number_of_deep_and_shallow},
if $s=k=\ell=1$, we get a graph with vertex set
$\{u_1,w_1,v_1,v_1'\}$ which is the same graph as a star with centre
$u_1$:
$$
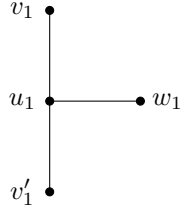
\begin{tikzpicture}[scale=0.30]
\filldraw(0,4) circle (5pt);
\node at (-1.2,4) {$v_1$};
\filldraw(0,0) circle (5pt);
\node at (-1.2,0) {$u_1$};
\filldraw(0,-4) circle (5pt);
\node at (-1.2,-4) {$v_1'$};
\filldraw(4,0) circle (5pt);
\node at (5.2,0) {$w_1$};
\draw (0,4) -- (0,-4);
\draw (0,0) -- (4,0);
\end{tikzpicture}
$$
Now consider the generalized Laplacian, where we write the matrices in
the vertex order: $u_1,w_1,v_1,v_1'$:
$$
M_0 =
\begin{bmatrix} 
1 & -1 & -1 & -1 \\
-1 & -1 & 0 & 0 \\
-1 & 0 & 0 & 0 \\
-1 & 0 & 0 & 0 
\end{bmatrix}
$$
We easily check that the characteristic polynomial of $M_0$ is
$$
p(\lambda)=\det(\lambda I - M) = \lambda (\lambda^3-4\lambda-2),
$$
and since $\lambda^3-4\lambda-2$ maps $-\infty\mapsto-\infty$, $-1\mapsto 2$
and $0\mapsto -2$, we see that root $\lambda=0$ of $p(\lambda)$ is the
third smallest eigenvalue, i.e., $\lambda_3(M)=0$,
and as in Section~\ref{se_shallow_deep}
has corresponding eigenvector $\mec f_3=(0,0,1,-1)$
(so $u_1$ is a shallow Urschel vertex of $\mec f_3$
and $w_1$ is a deep Urschel vertex).
Letting 
$$
N = 
\begin{bmatrix} 
0 & 0 & -1/2 & -1/2 \\
0 & -1 & 1/2 & 1/2 \\
-1/2 & 1/2 & -1/2 & -1/2 \\
-1/2 & 1/2 & -1/2 & -1/2 
\end{bmatrix}
$$
we easily see that $MN$ takes $\mec f_3$ to $\mec 0$ and is the identity
on $\mec f_3^\perp$.  Hence $N=M_0^+$.
Taking $M_1={\rm diag}(0,0,0,1)$, we have
$$
\mec f_3^1 = M_0^+(-M_1\mec f_3)=M_0^+
\begin{bmatrix} 0 \\ 0 \\ 0 \\ 1 \end{bmatrix}
=
\begin{bmatrix} -1/2 \\ 1/2 \\ -1/2 \\ -1/2 \end{bmatrix}
$$
which has opposite signs on
$u_1$ and $w_1$.  Hence, although
$$
{\rm WND}(\mec f_3)={\rm UN}(\mec f_3)={\rm SND}(\mec f_3)=2,
$$
for a small perturbation $M(\epsilon)=M_0+\epsilon M_1$, the
third eigenvector is simple and --- for $\epsilon\ne 0$ and sufficiently
small --- has opposite signs on $u_1$ and $w_1$.
Hence we easily see that the perturbed eigenvector 
$\mec f_3(\epsilon)=\mec f_3 + \epsilon \mec f_3^1 + O(\epsilon^2)$
has, for small $\epsilon\ne 0$
$$
{\rm WND}(\mec f_3(\epsilon))={\rm UN}(\mec f_3(\epsilon))
={\rm SND}(\mec f_3(\epsilon))=3.
$$

\subsection{Our Perturbations At a Multiple Eigenvalue}

Although we know that generically a generalized Laplacian has simple
eigenvalues, the perturbations near a multiple eigenvalue
in this article are not ``generic perturbations.''
Indeed, in Section~\ref{se_perturbation}, 
Lemma~\ref{le_distinct_eigenvalues} considers only diagonal perturbation
matrices; moreover, we insist that $d_{i+1}> C d_i$ and we restrict
ourselves to only $d_1,\ldots,d_m$ being nonzero.
However, once we find appropriate $M_1$ in
Lemma~\ref{le_reduce}, then we can replace $M_1$ with $\tilde M_1$
as long as $\|M_1-\tilde M_1\|_{L^2}$ is strictly less than the
minimum eigenvalue gap of the eigenvalues of $PM_1P$.
So in Lemma~\ref{le_reduce}, if $PM_1P$ has distinct eigenvalues
on $E$, then so does any $P\tilde M_1 P$ for $\tilde M_1$ in a 
neighbourhood of $M_1$.

Notice that if $M_0$ has more than one multiple eigenvalue, we can
find an $M_1$ such that $M(\epsilon)=M_0+\epsilon M_1$ has all
eigenvalues being simple.  In other words, we claim that 
in Lemma~\ref{le_reduce}, if $E_1,\ldots,E_\ell$
are mutually orthogonal subspaces of $\reals^m$, and $P_i$ is the
orthogonal projection onto $E_i$, then we can find
a single $M_1$ such that $P_i M_1 P_i$ has distinct eigenvalues on
$E_i$ for all $i\in[\ell]$.
To prove this, we give the same argument, but where
$\mec u_1,\ldots,\mec u_m$ in the proof of Lemma~\ref{le_reduce} is
replaced by the union of orthonormal bases for each $E_i$
and we repeat the same argument.

Finally, we note that in analysis it is well-known that in the
analog of Lemma~\ref{le_perturbation_at_multiple_eigenvalue_main},
the $\mec f_k^{1,0},\ldots,\mec f_k^{m,0}$ cannot be choosen a priori
(but, instead, are determined via the perturbation theory).
This is one reason that in perturbation theory in analysis,
most often one assumes that eigenvalues are
simple .
}

\subsection{Mild Improvements to the Gladwell-Zhu Result}

Gladwell-Zhu \cite{gladwell_zhu} showed a generalized
Laplacian, $M$, on a simple graph $G$, has a mutually orthogonal eigenbasis
$f_1,\ldots,f_n$ such that ${\rm SND}(f_k)\le k$ for all $k$.
Partially inspired by 
Theorem~\ref{th_perturbation_at_multiple_eigenvalue_main}, one 
can look for slight improvements of their results in the presence of
multiple eigenvalues.

In what follows, we use the formalism in 
\cite{friedman_geometric_aspects}: namely, we associate to a graph, $G$,
its {\em geometric realization}, $\cG$, where each edge of $G$ is replaced by
a unit length real interval.  The eigenvalues of $M$ can be viewed as
successive minimizers of the Rayleigh quotient
$\cR_M(\mec f)$ defined for functions $\mec f\from V_G\to\reals$ and given by
\begin{equation}\label{eq_Rayleigh_quotient_discrete}
\cR_M(\mec f) = \frac{
(M\mec f)\cdot\mec f
}{
\mec f\cdot\mec f
}
\end{equation} 
To get a ``geometric viewpoint,'' first we write:
$$
(M\mec f)\cdot\mec f  =  \sum_{ \{u,v\}\in E_G } 
\bigl( f(u)-f(v) \bigr)^2 \bigl( - M(u,v) \bigr) 
+ \sum_{v\in V} f(v)^2 C(v)
$$
where 
$$
C(v) = M(v,v) + \sum_{u\sim v} M(u,v) ;
$$
for example, if $M=\Delta_G$, then $C(v)=0$ for all $v\in V_G$.
Then
we consider all piecewise-linear
functions $f\from \cG\to\reals$ and define on such functions a Rayleigh
quotient
$$
\cR_M(f) = 
\frac{
\int_\cG |\nabla f|^2\,d\mu_2+
\int_\cG f^2\,d\mu_3
}{
\int_\cG f^2\,d\mu_1}
$$
where:
\begin{enumerate}
\item $d\mu_1$ is supported on $V_G$ and assigns a measure $1$ to each
$v\in V_G$;
\item $d\mu_3$ is supported on $V_G$ and assigns a measure $C(v)$ to each
$v\in V_G$ (hence this can be negative);
\item $d\mu_2$ is supported on the unit intervals in $\cG$
corresponding to the edges of $G$, and on each interval corresponding
to an edge $\{u,v\}$ equals 
$-M(u,v)$ times
the usual measure $dx$ on the unit interval.
\end{enumerate}
(Compare this to the first two formulas of Section~2 of
\cite{friedman_geometric_aspects}.)
Then the successive minimizers of $\cR_M(f)$ must be linear across
each edge, and the restrictions of each successive 
minimizer to $V_G$ is precisely
the successive minimizer of $\cR_M$ in
\eqref{eq_Rayleigh_quotient_discrete}.
A nodal region of a function $f\from\cG\to\reals$ now becomes a graph
with possibly edges of fractional length and one ``boundary vertex'' along
such an edge (compare to Section~2, \cite{friedman_geometric_aspects}).

We invite the reader who prefers to work with discrete nodal regions,
as is more common, to do so; see 
\cite{gladwell_zhu,davies_et_al2001,duval_reiner}, for example.

To understand our improvement to the Gladwell-Zhu result, 
say that --- for example ---
$\lambda_1 < \lambda_2=\lambda_3=\lambda_4$,
and that we have chosen $f_1,f_2,f_3,f_4$ such that
${\rm SND}(f_k)\le k$ for $k\le 4$.
We claim that either ${\rm SND}(f_4)\le 3$ or one can choose
$f_2',f_3',f_4'$ 
so that $f_1,f'_2,f_3',f'_4$ are mutually orthogonal eigenvectors but
${\rm SND}(f'_k)\le k$ for $2\le k\le 4$ and
${\rm SND}(f'_3)\le 2$:
indeed, if ${\rm SND}(f_4)=4$, then $f_4$ has 
4 strong nodal domains, say $R_1,R_2,R_3,R_4$.
Let $g_i$ be the restriction of $f_4$ to $R_i$, and extended by zero
outside $R_i$, using the geometric nodal domains of
\cite{friedman_geometric_aspects}.
We can choose
nonzero $f_2'=a_1 g_1+a_2 g_2$ to be orthogonal to $f_1$,
and nonzero $f_3'=a_3 g_3+a_4 g_4$ again orthogonal to $f_1$;
then automatically $f_2'$ is orthogonal to $f_3'$.
Hence ${\rm SND}(f_3')=2$.
We can choose nonzero $f'_4=b_1 g_1 + \cdots + b_4 g_4$ orthogonal to
$f_1,f'_2,f'_3$, and hence ${\rm SND}(f'_4)\le 4$ again.

Similarly, if $\lambda_2=\cdots=\lambda_6$, then either
${\rm SND}(f_6)\le 5$, or we could choose
$f'_2$ and $f'_3$ as before, and $f'_4=a_5 g_5+a_6 g_6$ with similar
notation, so that 
${\rm SND}(f'_k)=2$ for $k=2,3,4$.
Alternatively, we could choose:
$f'_2$ to be a combination of $g_1,g_2$,
$f'_3$ to be a combination of $g_4,g_5$,
$f'_4$ of $g_1,g_2,g_3$
and $f'_5$ of $g_4,g_5,g_6$.
Then we have
$$
{\rm SND}(f'_2)={\rm SND}(f'_3)=2,
\ {\rm SND}(f'_4)={\rm SND}(f'_5)=3,
$$
which more closely resembles 
Theorem~\ref{th_perturbation_at_multiple_eigenvalue_main}
(up to some point).

We haven't explored the extent to which one can improve the above result.
However, note that if $\lambda_2<\lambda_3=\lambda_4<\lambda_5$,
then we cannot improve over ${\rm SND}(\mec f_k)\le k$ for $k=3,4$.
Hence the above type improvements don't work for
higher eigenvalues unless its multiplicity is proportionally as 
large
(e.g., you can get some improvement when $\lambda_3$ has multiplicity 4).

\appendix
\section{Details of Eigenvalue Calculations in Section~4}
\label{ap_details_of_ladders_and_split_paths}

In this section we give the details of the eigenvalues computations
claimed in Section~\ref{se_urschel_examples}.

\subsection{Details of A Simple Ladder with (Possibly) Multiple Middle Rungs}
\label{su_details_of_ladder_with_multiple_middle_rungs}

In this subsection we give all the eigenvalues of 
the Laplacian of the graph in 
Subsection~\ref{su_simple_ladder_with_multiple_middle_rungs}.

It will be simpler, at times, to work with $j=n-6$, which is the number
of middle vertices.

We claim that the eigenvalues are
$$
0=\lambda_1 < \lambda_2 < 1=\lambda_3=\lambda_4 <
\lambda_5 < 2=\lambda_6=\cdots = \lambda_{n-2} < \lambda_{n-1}
\le \lambda_n ,
$$
meaning that the eigenvalue $2$ only occurs if $n\ge 8$.  To
see this we observe:
\begin{enumerate}
\item
We have $\lambda_1=0<\lambda_2$ since $G$ is connected.
\item
We have demonstrated above that $1$ is an eigenvalue that occurs with
multiplicity at least $2$, by the functions in
\eqref{eq_eigenvalue_one_for_ladder_with_multiple_rungs}.
\item
The functions
$$
\begin{tikzpicture}[scale=0.20]
\node at (-22,0) {$\forall a,b\in\reals$,};
\tikzmath{ 
  int \j ;
}
\foreach \i in {1,...,3} {
  \filldraw (0,6 - \i * 3) circle (5pt);
  \filldraw (10,6 - \i * 3) circle (5pt);
  \tikzmath{
    \j = \i + 3;
  }
  \filldraw (5, 2 - \i * 1 ) circle (3pt);
}
\node[anchor=east]  at (-1,3) {\Small $f(v_1)=a$};
\node[anchor=east] at (-1,0) {\Small $f(v_2)=b$};
\node[anchor=east] at (-1,-3) {\Small $f(v_3)=a$};
\node[anchor=west] at (11,3) {\Small $f(v_4)=-a$};
\node[anchor=west] at (11,0) {\Small $f(v_5)=-b$};
\node[anchor=west] at (11,-3) {\Small $f(v_6)=-a$};
\draw (0,3) to (0,-3);
\draw (10,3) to (10,-3);
\filldraw (5,5) circle (5pt);
\node at (5,6.5) {\Small $f(v_7)=0$};
\filldraw (5,-5) circle (5pt);
\node at (5,-6.5) {\Small $f(v_n)=0$};
\draw (0,0) to (5,5);
\draw (10,0) to (5,5);
\draw (0,0) to (5,-5);
\draw (10,0) to (5,-5);
\filldraw (5,3) circle (5pt);
\filldraw (5,-3) circle (5pt);
\draw (0,0) to (5,3);
\draw (10,0) to (5,3);
\draw (0,0) to (5,-3);
\draw (10,0) to (5,-3);
\end{tikzpicture}
$$
The equation $\Delta_G f=\lambda f$ for $f$ as above requires, at $v_1$
and $v_2$ respectively:
\begin{align*}
a-b & = \lambda a \\
(j+2) b -2a & = \lambda b
\end{align*}
and hence
$$
p(\lambda)=\lambda^2-(3+j)\lambda+j = 0.
$$
Since $p(0)=-j<0$, $p(1)=-1<0$, one solution lies in the open
interval $(0,1)$.  Since $p(3)=-2j<0$,
and $p(\lambda)\to\infty$ as $\lambda\to\infty$, we have
one solution lies satisfies $\lambda>3$.
\item
Similarly,
the functions
$$
\begin{tikzpicture}[scale=0.20]
\node at (-22,0) {$\forall a,b,c\in\reals$,};
\tikzmath{ 
  int \j ;
}
\foreach \i in {1,...,3} {
  \filldraw (0,6 - \i * 3) circle (5pt);
  \filldraw (10,6 - \i * 3) circle (5pt);
  \tikzmath{
    \j = \i + 3;
  }
  \filldraw (5, 2 - \i * 1 ) circle (3pt);
}
\node[anchor=east]  at (-1,3) {\Small $f(v_1)=a$};
\node[anchor=east] at (-1,0) {\Small $f(v_2)=b$};
\node[anchor=east] at (-1,-3) {\Small $f(v_3)=a$};
\node[anchor=west] at (11,3) {\Small $f(v_4)=a$};
\node[anchor=west] at (11,0) {\Small $f(v_5)=b$};
\node[anchor=west] at (11,-3) {\Small $f(v_6)=a$};
\draw (0,3) to (0,-3);
\draw (10,3) to (10,-3);
\filldraw (5,5) circle (5pt);
\node at (5,6.5) {\Small $f(v_7)=c$};
\filldraw (5,-5) circle (5pt);
\node at (5,-6.5) {\Small $f(v_n)=c$};
\draw (0,0) to (5,5);
\draw (10,0) to (5,5);
\draw (0,0) to (5,-5);
\draw (10,0) to (5,-5);
\filldraw (5,3) circle (5pt);
\filldraw (5,-3) circle (5pt);
\draw (0,0) to (5,3);
\draw (10,0) to (5,3);
\draw (0,0) to (5,-3);
\draw (10,0) to (5,-3);
\end{tikzpicture}
$$
The equation $\Delta_G f=\lambda f$ for $f$ as above requires, at 
$v_1,v_2,v_7$ respectively:
\begin{align*}
a-b & = \lambda a, \\
-2a +(j+2) b -jc & = \lambda b, \\
-2b +2c & = \lambda c .
\end{align*}
Since this also includes the solution $\lambda=0$ and $a=b=c=1$,
we can choose $a,b,c$ to make $f$ orthogonal to $\mec 1$, and therefore
impose
$$
4a+2b+jc = 0.
$$
Adding this to the middle equation, the first two equations read:
\begin{align*}
a-b & = \lambda a, \\
2a +(j+4) b & = \lambda b, \\
\end{align*}
which amounts to
$$
q(\lambda)=\lambda^2-(5+j)\lambda+(j+6) = 0.
$$
Since $q(1)=2>0$ and $q(2)=4-(5+j)2+(j+6)=-j<0$ (and $q(\lambda)\to\infty$
as $\lambda\to\infty$), one solution satisfies $1<\lambda<2$ and
the other $\lambda>2$.
\item
Similarly,
the functions
$$
\begin{tikzpicture}[scale=0.20]
\node at (-22,0) {$\forall a_1+\cdots+a_j=0$,};
\tikzmath{ 
  int \j ;
}
\foreach \i in {1,...,3} {
  \filldraw (0,6 - \i * 3) circle (5pt);
  \filldraw (10,6 - \i * 3) circle (5pt);
  \tikzmath{
    \j = \i + 3;
  }
  \filldraw (5, 2 - \i * 1 ) circle (3pt);
}
\node[anchor=east]  at (-1,3) {\Small $f(v_1)=0$};
\node[anchor=east] at (-1,0) {\Small $f(v_2)=0$};
\node[anchor=east] at (-1,-3) {\Small $f(v_3)=0$};
\node[anchor=west] at (11,3) {\Small $f(v_4)=0$};
\node[anchor=west] at (11,0) {\Small $f(v_5)=0$};
\node[anchor=west] at (11,-3) {\Small $f(v_6)=0$};
\draw (0,3) to (0,-3);
\draw (10,3) to (10,-3);
\filldraw (5,5) circle (5pt);
\node at (5,6.5) {\Small $f(v_7)=a_1$};
\filldraw (5,-5) circle (5pt);
\node at (5,-6.5) {\Small $f(v_n)=a_j$};
\draw (0,0) to (5,5);
\draw (10,0) to (5,5);
\draw (0,0) to (5,-5);
\draw (10,0) to (5,-5);
\filldraw (5,3) circle (5pt);
\filldraw (5,-3) circle (5pt);
\draw (0,0) to (5,3);
\draw (10,0) to (5,3);
\draw (0,0) to (5,-3);
\draw (10,0) to (5,-3);
\end{tikzpicture}
$$
each give an eigenfunction with eigenvalue $2$.  Hence the
corresponding eigenspace has dimension $j-1=n-7$ (hence $2$ isn't an
eigenvalue in the case $n=7$).
\end{enumerate}
The reader can easily check that this process has produced 
five subspaces of functions (counting~(4) as a two-dimensional space,
having imposed $f$ is orthogonal to $\mec 1$), each of which is
orthogonal to all the others, and whose sums of dimensions is $n$.

\begin{remark}
At the risk of ``overkill,'' we have
$$
q(\lambda) - p(\lambda) = 6 - 2 \lambda,
$$
and it follows that at the larger root of $p(\lambda)$, which satisfies
$\lambda>3$ (since $p(3)<0$ as per above), $q$ at this root is negative;
hence the larger root of $q$ is larger than the larger root of $p$.
Moreover, one easily see that for $j\to\infty$ the roots of $p(\lambda)$
are
$$
\lambda = 1 - 2/j + O(1/j^2)= 1 - 2/n + O(1/n^2),
\quad
\lambda = j+ 2 + 2/j + O(1/j^2) = n-4 + 2/n + O(1/n^2),
$$
and those of $q(\lambda)$ are
$$
\lambda = 1 + 2/j + O(1/j^2)= 1 + 2/n + O(1/n^2),
\quad
\lambda = j+ 4 - 2/j + O(1/j^2) = n-2 - 2/n + O(1/n^2).
$$
\end{remark}

\subsection{Details of the Eigenvalue $1$ for a Path with Two Left Ends}
\label{su_path_with_two_left_ends_details}

In this subsection we outline a proof of the spectral results claimed in
Subsection~\ref{su_path_with_two_left_ends_details}.
So let $G_n$ be the $n$-vertex path with two left ends, as in
Subsection~\ref{su_path_with_two_left_ends}.

The diagrams in Subsection~\ref{su_path_with_two_left_ends}
visibly demonstrates a two-dimensional eigenspace for $\lambda=1$
when $n=3m+1$, and a one-dimensional eigenspace when $n=3m,3m-1$.
Hence $1$ is an Laplacian eigenvalue of $G_n$ of multiplicity at
least $2$ for $n=3m+1$ and multiplicity at least $1$.
Hence it suffices to prove that respective upper bounds on
the multiplicity of $1$ as a Laplacian eigenvalue of $G$ and to
find their indices as an eigenvalue.

By discarding one of the edges at one of the split ends of $G_n$ we get
a graph, $G'$, with two connected components: one that is an isolated vertex,
the other that is a path of $n-1$ vertices:
$$
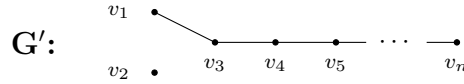
\begin{tikzpicture}[scale=0.20]
\node at (-8,0) {\Large\bf $\bf G'$:};
\foreach \i in {3,...,5} {
  \filldraw (4*\i-8,0) circle (5pt);
  \node at (4*\i-8,-1.5) {\Small $v_\i$};
}
\filldraw (0,2) circle (5pt);
\filldraw (0,-2) circle (5pt);
\node [anchor=east] at (-1,2) {\Small $v_1$};
\node [anchor=east] at (-1,-2) {\Small $v_2$};
\draw (0,2) to (4,0);
\draw (4,0) to (14,0);
\node at (16,0) {$\cdots$};
\node at (20,-1.5) {\Small $v_n$};
\draw (18,0) to (20,0);
\filldraw (20,0) circle (5pt);
\end{tikzpicture}
$$
The path on $n-1$ vertices, $P_{n-1}$,
is easily seen to have Laplacian eigenvalues $\lambda_1,\ldots,\lambda_{n-1}$
where
\begin{equation}\label{eq_lambda_j_for_path_eigenvalues}
\lambda_j\bigl(P_{n-1}\bigr) 
= 2 - 2 \cos\bigl(\pi (j-1)/(n-1) \bigr)
\end{equation} 
(and eigenfunction $\cos\bigl(\pi (j-1)(x-1/2)/(n-1)\bigr)$ with 
$x=1,\ldots,n-1$
being the vertices of the path).  Note that these eigenvalues are
strictly increasing as $j$ increases, and that
$$
\lambda_j(P_{n-1})=1 
\quad\iff\quad 
\pi(j-1)/(n-1)=\pi/3 
\quad\iff\quad 
j-1=(n-1)/3;
$$
hence also
$$
\lambda_j(P_{n-1})<1
\quad\iff\quad 
j-1<(n-1)/3
$$
and similarly with ``$<$'' replaced in both occurrences with ``$>$''.
It follows that:
\begin{enumerate}
\item
for $n=3m+1$, we have 
\begin{equation}\label{eq_n_equals_three_m_plus_one_path_eigs}
\lambda_{(n-1)/3}(P_{n-1})=\lambda_m(P_{n-1})=1
\end{equation}
(and $\lambda_j<1$ if $j<m$ and $\lambda_j>1$ if $j>m$); and
\item
for $n=3m$ and $n=3m-1$, we have (for all $j\in\integers$)
$j-1<(n-1)/3$ iff $j\le m$ (and otherwise $j-1>(n-1)/3$) and hence
\begin{equation}\label{eq_n_not_equals_three_m_plus_one_path_eigs}
\lambda_{m}(P_{n-1})<1<\lambda_{m+1}(P_{n-1}) .
\end{equation}
\end{enumerate}

Since $G'$ is $P_{n-1}$ plus an isolated vertex, its Laplacian 
has eigenvalues:
$$
\lambda_1(G')=0,\ \lambda_2(G')=\lambda_1(P_{n-1})=0<\cdots<
\lambda_n(G')=\lambda_{n-1}(P_{n-1}) .
$$
Since the Laplacian for $G$ is that of $G'$ plus a rank one 
positive semidefinite update
(for the single added edge), the min-max and max-min principle implies
\begin{equation}\label{eq_eigenvalue_intertwining}
0 \le \lambda_1(G_n) \le \lambda_1(P_{n-1}) \le \lambda_2(G_n) \le
\lambda_2(P_{n-1}) \le \cdots \le \lambda_{n-1}(P_{n-1})\le \lambda_n(G_n).
\end{equation} 
It follows that:
\begin{enumerate}
\item
for $n=3m$ and $n=3m-1$, 
\eqref{eq_eigenvalue_intertwining} 
and \eqref{eq_n_not_equals_three_m_plus_one_path_eigs} implies that 
$1$ is an eigenvalue
of $G_n$ of multiplicity
at most $1$, and if so then $1=\lambda_{m+1}(G_n)$;
\item
for $n=3m+1$, 
\eqref{eq_eigenvalue_intertwining}
and \eqref{eq_n_equals_three_m_plus_one_path_eigs}
implies that 
$1$ is a Laplacian eigenvalue
of $G_n$ of multiplicity
at most $2$, and if so then $1=\lambda_{m+1}(G_n)=\lambda_{m+2}(G_n)$.
\end{enumerate}
But these upper bounds on the multplicity of $1$ as a Laplacian eigenvalue
of $G_n$ are matched by the lower bounds mentioned at the beginning
of this section.  Hence we get equality for all $n$, and the
items above also locate the index of $1$ as an eigenvalue of $G_n$.


\newcommand{\etalchar}[1]{$^{#1}$}
\providecommand{\bysame}{\leavevmode\hbox to3em{\hrulefill}\thinspace}
\providecommand{\MR}{\relax\ifhmode\unskip\space\fi MR }
\providecommand{\MRhref}[2]{%
  \href{http://www.ams.org/mathscinet-getitem?mr=#1}{#2}
}
\providecommand{\href}[2]{#2}

\end{document}